\documentclass[jmp,12pt]{revtex4-1}
\usepackage{amsmath}
\usepackage{amssymb}
\usepackage{amsthm}
\usepackage{graphicx}
\usepackage{subfig}
\usepackage{dcolumn}
\usepackage{bm}

\newtheorem{thm}{Theorem}
\newtheorem{lem}{Lemma}

\newcommand{\beq}{\begin{equation}}
\newcommand{\eeq}{\end{equation}}

\newcommand{\ba}{\begin{array}}
\newcommand{\ea}{\end{array}}

\newcommand{\bea}{\begin{eqnarray}}
\newcommand{\eea}{\end{eqnarray}}

\newcommand{\bc}{\begin{center}}
\newcommand{\ec}{\end{center}}

\newcommand{\ds}{\displaystyle}

\newcommand{\bt}{\begin{tabular}}
\newcommand{\et}{\end{tabular}}

\newcommand{\bi}{\begin{itemize}}
\newcommand{\ei}{\end{itemize}}

\newcommand{\bd}{\begin{description}}
\newcommand{\ed}{\end{description}}

\newcommand{\p}{\partial}

\newcommand{\cf}{{\it cf.}~}

\newcommand{\norm}[1]{\left|\left|#1\right|\right|_{2}}

\newcommand{\Lnpnorm}{L_{2}\left(S_{T}\right)}

\newcommand{\Hpnorm}{H_{1}(S_{T})}
\newcommand{\Hpsnorm}{H_{s}\left(S_{T}\right)}

\newcommand{\opHnormo}[1]{\left|\left|#1\right|\right|_{\Hpnorm}}
\newcommand{\opHnorm}[1]{\left|\left|#1\right|\right|_{\Hpsnorm}}

\newcommand{\gnorm}[1]{\left| \left| #1 \right| \right|}

\begin{document}
\title{On Nonlocal Gross-Pitaevskii Equations with Periodic Potentials}
\date{\today}
\author{Christopher W. Curtis}
\affiliation{Department of Applied Mathematics, University of Colorado, Boulder, CO 80309, USA.}
\email{christopher.w.curtis@colorado.edu}

\begin{abstract}
The Gross-Pitaevskii equation is a widely used model in physics, in particular in the context of Bose-Einstein condensates.  However, it only takes into account local interactions between particles.  This paper demonstrates the validity of using a nonlocal formulation as a generalization of the local model.  In particular, the paper demonstrates that the solution of the nonlocal model approaches in norm the solution of the local model as the nonlocal model approaches the local model.  The nonlocality and potential used for the Gross-Pitaevskii equation are quite general, thus this paper shows that one can easily add nonlocal effects to interesting classes of Bose-Einstein condensate models.  Based on a particular choice of potential for the nonlocal Gross-Pitaevskii equation, we establish the orbital stability of a class of parameter-dependent solutions to the nonlocal problem for certain parameter regimes.  Numerical results corroborate the analytical stability results and lead to predictions about the stability of the class of solutions for parameter values outside of the purview of the theory established in this paper. 
\end{abstract}
\keywords{Gross-Pitaevskii Equations, Nonlocal Models, Stability}
\maketitle
\section{Introduction}
The last 15 years has a seen a rapid growth in interest concerning the modeling of Bose-Einstein condensates.  The body of literature concerning this subject is too vast to consider here, but a simplified description of the field would include the study of the Gross-Pitaevskii equation
\beq
i\p_{t}\psi = -\frac{1}{2}\p^{2}_{x}\psi +\alpha |\psi|^{2}\psi + V(x)\psi,
\label{locgp}
\eeq
where $\alpha = \pm 1$, with $+1$ corresponding to repulsive interactions between particles in the condensate, and $-1$ corresponding to attractive interactions.  The function $\psi$ represents an approximation to the wave function used to describe the probability density for the location of particles in the condensate.  

The validity of this equation as an approximation to the many-particle formulation of the problem has been established in \cite{lieb}.  However, an assumption of a pairwise $\delta$-function interaction among particles is used to derive \eqref{locgp}.  This clearly cannot capture all of the  physics in the problem since each particle in the condensate exerts forces that act at a distance.  Thus the next order of approximation to the many-particle formulation would be to include a more general interaction potential simulating nonlocal interactions between particles.  This is done in \cite{decon} by studying the modified one-dimensional Gross-Pitaevskii equation
\beq
i\p_{t}\psi = -\frac{1}{2}\p^{2}_{x}\psi +\alpha\psi(x)\int_{-\infty}^{\infty}R(x-y;\epsilon)|\psi|^{2}(y)dy + V(x)\psi,
\label{nlocgp}
\eeq
where $\epsilon > 0$, and $R(x;\epsilon) = \frac{1}{\epsilon}\zeta(\frac{x}{\epsilon})$, with $\zeta(x)$ being a positive, even function such that
\beq
\lim_{\epsilon \rightarrow 0^{+}} R(x;\epsilon) = \delta(x) \nonumber
\eeq
in the sense of distributions.  In \cite{decon}, $\epsilon$ is called the nonlocality parameter.  The authors of \cite{decon} assume that the condensate is trapped in both a harmonic confining potential and an external standing-wave potential.  While in \cite{decon} a three-dimensional version of \eqref{nlocgp} is derived, the presence of the standing-wave potential allows the reduction to a one-dimensional model (\cf \cite{olshan}).  

Nonlocal models like \eqref{nlocgp} are also called Hartree-Fock equations.  These have been extensively studied in the case that $\zeta(x) = 1/|x|$, {\it i.e.} in modeling Coulombic interactions between particles (\cf \cite{gini}, \cite{hart}).  Recent literature on the formation of dipolar condensates has introduced nonsingular nonlocalities characterized by cubic decay (\cf \cite{sinha}, \cite{kevre}).  These nonlocalities with cubic decay fit into the class studied in this paper.  Other models with a varying nonlocality parameter have appeared in the optics literature \cite{trillo}.  The analysis of the well-posedness and convergence of nonlocal, nonlinear Schr\"{o}dinger type models to local ones can be found in \cite{cao1} and \cite{cao2}, though the models examined in those papers are different from those studied in this paper. 

The authors of \cite{decon}, working with the potential $V(x) = V_{0}\sin^{2}(kx)$, derived the traveling-wave solutions
\beq
\psi(x,t) = r_{sol}(x)e^{i\theta(x)-i\omega t},
\label{sol}
\eeq
where
\begin{eqnarray}
r_{sol}^{2}(x)= & B - \ds{\frac{V_{0}}{\alpha \beta(k;\epsilon)}}\sin^{2}(kx), \nonumber \\
\tan(\theta(x))= & \ds{\sqrt{1-\ds{\frac{V_{0}}{\alpha B \beta(k;\epsilon)}}}}\tan(kx), \nonumber \\
\omega(k)= & \ds{\frac{V_{0}+k^2}{2}} + \alpha B - \ds{\frac{V_{0}}{2\beta(k;\epsilon)}}, \nonumber \\
\beta(k;\epsilon)= & \ds{\int_{-\infty}^{\infty}} R(x;\epsilon) \cos(2kx) dx, \nonumber
\end{eqnarray}
with $B$ a constant called the offset size.  Defining $A = \frac{-V_{0}}{\alpha \beta(k;\epsilon)}$, the traveling-wave solution can be rewritten as
\beq
\psi(x,t) = \left(\sqrt{B}\cos(kx) + i\sqrt{B+A}\sin(kx) \right)e^{-i\omega t}, \nonumber
\eeq   
which shows the spatial component of \eqref{sol} is periodic with period $2\pi/k$.  The coefficients appearing in $\psi(x,t)$ must satisfy the restrictions 
\[
 B \geq \max\{-A,0\}, ~ \alpha = \pm 1, ~   \beta(k;\epsilon) \neq 0.
\]
 
Setting $\alpha$ and $k$ equal to one, and taking $B=1$, which in \cite{decon} is described as large, and $V_{0}=-1$, the authors of \cite{decon} study the stability of \eqref{sol} by numerical simulations using $\zeta(x) = e^{-x^2}$ and $V(x)=V_{0}\sin^{2}(kx)$.  The authors report results which numerically demonstrate that for the local case, {\it i.e.} when $\epsilon = 0$, \eqref{sol} is stable with respect to perturbations due to roundoff error in the numerical simulation.  However, their results also suggest that \eqref{sol} is unstable when the nonlocality parameter $\epsilon$ is positive, and that the instability emerges at a fixed time in their simulations, independent of the value of $\epsilon$.  The authors of \cite{decon} also study the effect of changing the convolution kernel, and they report that the results are similar to those for the case $\zeta(x)=e^{-x^{2}}$.  

It is conjectured in \cite{decon} that a beyond-all-orders phenomena may be responsible for the behavior exhibited in their numerics.  As pointed out in \cite{decon}, if the behavior exhibited in their numerics is accurate and truly independent of the choice of interaction potential, then \eqref{nlocgp} cannot be viewed as a valid generalization of \eqref{locgp}.  That is to say, no matter how small one makes the nonlocal interaction term, the results of \cite{decon} seem to imply that one cannot approach the local behavior.  This is described as a lack of {\it asymptotic equivalence of stability} (AES) in \cite{decon}.     

The purpose of this paper is to address both the issue of whether or not \eqref{nlocgp} is AES to \eqref{locgp} and under what conditions \eqref{sol} is a stable solution of Equation \eqref{nlocgp}.  To do this, we first fix some notation and introduce the spaces in which we work.  Let $S_{T}$ denote the circle of circumference $T$.  Introduce the space $L^{2}(S_{T})$ which is the completion of the continuous $T$-periodic functions in the norm
\beq
\gnorm{f}_{L^{2}(S_{T})} = \left(\int_{S_{T}}|f(x)|^{2}dx \right)^{1/2} = \left(\int^{T/2}_{-T/2}|f(x)|^{2}dx\right)^{1/2}. \nonumber
\eeq  
In practice the integral $\int_{S_{T}}|f(x)|^{2}dx$ could be evaluated over any interval of width $T$ since $f$ is a $T$-periodic function.  Note, throughout the remainder of the text $\gnorm{\cdot}_{L^{2}(S_{T})}$ is abbreviated by $\norm{\cdot}$.  We  define the norm, denoted as $\gnorm{\cdot}_{2,v}$, of the product space $L^{2}_{2}(S_{T}) = L^{2}(S_{T}) \times L^{2}(S_{T})$, via 
\[
\gnorm{\left(\ba{c}f \\ g \ea\right)}^{2}_{2,v} = \gnorm{f}^{2}_{2}  + \gnorm{g}^{2}_{2}. 
\]
For operators $\mathcal{B}$ that map $L^{2}(S_{T})$ to itself, we denote the norm of $\mathcal{B}$ via
\[
\norm{\mathcal{B}} = \sup_{\norm{f}=1} \norm{\mathcal{B}f}.
\]
The norms $\gnorm{\mathcal{B}}_{2,v}$ are defined in an identical way.  The Sobolev spaces $H_{s}(S_{T})$ are defined as follows:
\[
H_{s}\left(S_{T}\right) = \left\{f \in L^{2}(S_{T}): \sum_{j=-\infty}^{\infty}\left<j\right>^{s}|\hat{f}_{j}|^{2} <\infty \right\}, \nonumber
\]
where 
\[
\left<j\right> = \left(1+\frac{4\pi^2j^{2}}{T^{2}}\right),
\]
and the terms $\hat{f}_{n}$ come from the Fourier series of $f(x)$, which is 
\beq
f(x) = \sum_{j=-\infty}^{\infty} \hat{f}_{j} e_{j}(x), \nonumber
\eeq
where 
\beq
e_{j}(x) = 
\left\{ 
\ba{cc}
\ds{\frac{1}{T}} & j=0 \\
\\
\ds{\frac{1}{\sqrt{T}}}e^{-\ds{2\pi i j x/T}} & j \neq 0
\ea
\right. 
\label{fs}
\eeq
and 
\beq
\hat{f}_{j} = \int_{-T/2}^{T/2} f(x) e^{*}_{j}(x) dx, \nonumber
\eeq
where $e^{*}_{j}$ denotes the complex conjugate of $e_{j}$.  The product space $H_{s}(S_{T})\times H_{s}(S_{T})$ is denoted by $H^{2}_{s}(S_{T})$.  Finally, define the Fourier transform of $h(x) \in L^{2}(\mathbb{R})$, say $\hat{h}(\tilde{s})$, by
\beq
\hat{h}(\tilde{s}) = \int_{-\infty}^{\infty} e^{-i\tilde{s}x}h(x) dx. \nonumber
\eeq

To address the issue of AES, we first prove the local-well posedness, in $H_{s}(S_{T})$ for $s>1/2$, and global-well posedness of \eqref{nlocgp} over the space $H_{1}(S_{T})$ based on the following assumptions.
\begin{itemize}
\item[$H1$:] The potential $V(x)$ is a smooth, $T$-periodic function,
\item[$H2$:] $\zeta(x) \geq 0$,
\item[$H3$:] $\zeta(x) \in L^{1}(\mathbb{R})$ with $\gnorm{\zeta}_{L^{1}(\mathbb{R})}=\int_{\mathbb{R}}\zeta(x)dx = 1$, 
\item[$H4$:] $x\zeta(x) \in L^{1}(\mathbb{R})$, and 
\item[$H5$:] $|\hat{\zeta}(\tilde{s})| \leq (1+|\tilde{s}|)^{-1/2-\tilde{\epsilon}}$, where $\tilde{\epsilon}>0$.  
\end{itemize}
Note that the maximum of $\hat{\zeta}$ could be chosen larger than one without affecting our results.  We make this choice in order for a cleaner presentation.  Using the local and global-well posedness results, we prove     
\begin{thm}
\label{aes}
Let $\alpha =1$.  Assuming the hypotheses $H1-H5$, choose constant $C>0$ such that $\opHnormo{\psi_{0,\epsilon}}\leq C$ for $\epsilon \geq 0$ where $\psi_{0,\epsilon}(x)$ is an initial condition for \eqref{nlocgp} and $\psi_{0,0}(x)$ is an initial condition for \eqref{locgp}.  Let $\psi_{0,\epsilon} \rightarrow \psi_{0,0}$ in the $H_{1}(S_{T})$ norm as $\epsilon \rightarrow 0^{+}$.  Let $\psi(x,t)$ and $\psi_{\epsilon}(x,t)$ be the unique T-periodic solutions to \eqref{locgp} and \eqref{nlocgp} respectively for the given initial conditions.  Then there exists a constant $C_{u}>0$ and a function $\varrho(\tilde{T};\epsilon,\epsilon^{'})$, where $\lim_{\epsilon \rightarrow \epsilon^{'}}\varrho(\tilde{T};\epsilon,\epsilon^{'}) = 0$ for $\epsilon,\epsilon^{'}\geq 0$, such that, for any finite $\tilde{T}>0$, we have the bound
\[
\gnorm{\psi_{\epsilon}(\cdot,t) - \psi(\cdot,t)}_{L^{\infty}(S_{T})} \leq \left(\gnorm{\psi_{0,\epsilon}- \psi_{0,0}}_{H_{1}(S_{T})} + 3C_{u}\tilde{T}\varrho(\tilde{T};\epsilon,0)\right)e^{3C^{2}_{u}\tilde{T}}
\]
for $t\in[0,\tilde{T}]$.
\end{thm}  
Thus, on any finite interval of time, as one lets the nonlocality parameter approach zero, the solution to \eqref{nlocgp} converges uniformly in space to the solution of \eqref{locgp}.  This shows that AES is a common feature for a large class of potentials and nonlocal, repulsive interactions.  Therefore the results in \cite{decon} are likely due to artifacts of their numerical computations, as opposed to being inherent to the equation.     

As to the stability of \eqref{sol}, we first need to define the notion of stability to be established ({\it cf.} \cite{gss}).  Let $\psi(x,t)$ denote a solution to either \eqref{locgp} or \eqref{nlocgp} with initial condition $\psi(x,0)$.  Writing \eqref{sol} as $\phi_{\omega}(x)e^{-i\omega t}$, we say that $\eqref{sol}$ is orbitally stable in $H_{1}(S_{T})$, if for any $\rho>0$, there is a $\delta>0$ such that if $\left|\left|\psi_{0}(x) - \phi_{\omega}(x)\right|\right|_{H_{1}(S_{T})} < \delta$ then
\[
\sup_{t>0}\inf_{c \in [0,2 \pi)}\left|\left|\psi(x,t) - \phi_{\omega}(x)e^{ic}\right|\right|_{H_{1}(S_{T})} < \rho.
\]

The other notion of stability we use is that of spectral stability.  First, separate \eqref{locgp} or \eqref{nlocgp} into real and imaginary parts.  Denote the linearization of either of these systems around \eqref{sol} as $JL$.  Using the scaling $x\rightarrow kx$, $JL$ has terms that are $2\pi$-periodic functions.  Let $\sigma(JL)$ denote the spectrum of $JL$ computed over the space $H^{2}_{2}(S_{2\pi n})$, $n\in\mathbb{N}$.  In effect, we are computing the impact of perturbing \eqref{sol} by $2\pi n$-periodic perturbations, or $2\pi n/k$-periodic perturbations in the unscaled coordinate.  We say \eqref{sol} is spectrally stable if for $\lambda \in \sigma(JL)$, $\mbox{Re}(\lambda) \leq 0$.  Note, more details are provided in Section 3.  Also, given that the nonlinear problem is Hamiltonian, the condition of spectral stability reduces to having spectrum only on the imaginary line, {\it i.e.} $\mbox{Re}(\lambda)=0$.  With these definitions in hand, we prove the following three theorems.  Throughout these remaining theorems we assume that
\begin{itemize}
\item[$H1^{'}$:] $V(x)=V_{0}\sin^{2}(kx)$,
\item[$H2^{'}$:] $\zeta(x)\geq0$, $\zeta(x)$ is even, and $\gnorm{\zeta}_{L^{1}(\mathbb{R})}=1$,
\item[$H3^{'}$:] $\hat{\zeta}> 0$, and
\item[$H4^{'}$:] $\hat{\zeta}(\tilde{s}) \leq (1+|\tilde{s}|)^{-1/2-\tilde{\epsilon}}$, with $\tilde{\epsilon}>0$.
\end{itemize}

\begin{thm}
Let $\alpha=1$.  Assuming the Hypotheses $H1^{'}-H4^{'}$, for any values of $k$ and the nonlocality parameter $\epsilon$, and for perturbations of period $\frac{2\pi n}{k}$, where $n\in \mathbb{N}$, the solution \eqref{sol} is spectrally stable for sufficiently large offset size $B$, $V_{0}<0$, and $|V_{0}|$ sufficiently small.  
\label{specsta}
\end{thm}   

\begin{thm}
\label{orbsta}
Let $\alpha=1$.  Fix the nonlocality parameter $\epsilon$ and the value $k$.  Assuming the Hypotheses $H1^{'}-H4^{'}$, for offset parameter $B$ sufficiently large, $V_{0}<0$, and $|V_{0}|$ sufficiently small, the solution \eqref{sol} of the nonlocal Gross-Pitaevskii equation \eqref{nlocgp} is orbitally stable with respect to perturbations with periods $T=\frac{2\pi n}{k}$, where $n\in\mathbb{N}$.  
\end{thm}

\begin{thm}
\label{unstab}
Let $\alpha =1$.  Assuming the Hypotheses $H1^{'}-H4^{'}$, with $A=\frac{-V_{0}}{\alpha \beta(k;\epsilon)}$, if $A \geq 2.46 k^2$, then for offset size $B$ and $\epsilon$ sufficiently small, \eqref{sol} is spectrally unstable with respect to perturbations of period $\frac{2\pi n}{k}$, where $n\in \mathbb{N}$.  
\end{thm}

The content of these three theorems shows that the role of a small nonlocality parameter is dependent upon the other parameters in the problem, particularly the offset size $B$.  Theorems \ref{specsta} and \ref{orbsta} are proven by showing that if $B$ is sufficiently large, then the operator $L$ is positive semi-definite.  Then, using Krein signature arguments found in \cite{hara}, we get both spectral and orbital stability.  Thus, introducing small nonlocality should not effect the stability of \eqref{sol}, while Theorem \ref{unstab} shows that if $B$ is too small, then even removing the nonlocality parameter does not stabilize the solution.  In contrast, as is shown later, we can always get a spectrally stable problem by letting $\epsilon \rightarrow \infty$ for any choice of the other parameters.  Thus it appears that while a small amount of nonlocality does not affect stability, large amounts do.     

The above theorems do not allow for arbitrary choices of parameters since each theorem requires $|V_{0}|$ to be small, which ensures that $L$ remains positive semi-definite.  We cannot at this time provide explicit bounds on how large $|V_{0}|$ can be such that Theorems \ref{specsta} and \ref{orbsta} remain true since we cannot control the spectra of $JL$ for $V_{0}\neq 0$.  Therefore, we must treat the parameter values used in \cite{decon} as outside the scope of what is proved in this paper.  We provide numerical experiments in order to make conjectures about the stability of \eqref{sol}.  First, we use the above theorems to calibrate our numerics by picking parameter values that can reasonably be believed to satisfy the constraints of Theorems \ref{orbsta} and \ref{unstab}.  Our numerics behave as the theory predicts.  Second, we present numerical experiments using the parameter values found in \cite{decon}, and from this we conjecture that in fact \eqref{sol} should be stable for the parameter values chosen.  As mentioned above, these are $B=1$, $V_{0}=-1$, $k=1$, and $\alpha=1$.  

As in \cite{decon}, a pseudo-spectral method is used for the spatial variable, while a standard Runge-Kutta method is used for time evolution.  The most likely explanation for the discrepancy between the results reported here and those of \cite{decon} is the way in which the convolution is handled.  In this paper, no approximation is made in the integral or to the convolution kernel.  However, in \cite{decon}, it appears an approximation is made to the kernel which introduces an error that appears difficult for the pseudo-spectral method to resolve.  In private communications, the authors of \cite{decon} have been made aware of these discrepancies.  They have encouraged the explanation for them in this manuscript.  

The structure of the paper is as follows.  In Section 2, we present the proof of the AES of \eqref{locgp} and \eqref{nlocgp}.  In Section 3, we find the linearization around \eqref{sol}, and we establish some basic results about the convolution kernel that are used later.  In Section 4, Theorems \ref{specsta} and \ref{orbsta} are proved, while in Section 5, Theorem \ref{unstab} is proved.  Finally, Section 6 presents the numerical results.  
\section{Asymptotic Equivalence of Stability}
We proceed in the following fashion.  In order to make the presentation self-contained, we first establish the local-in-time well-posedness of the nonlocal Gross-Pitaevskii equation from which we obtain a local-in-time form of AES.  We then establish the continuity in $\epsilon$ of solutions to \eqref{nlocgp}.  Finally, we establish the global-in-time well posedness of \eqref{nlocgp} which allows us to prove Theorem \ref{aes}.  Note, we use Hypotheses $H1-H5$ throughout the remainder of the section. 

We begin by establishing some basic lemmas concerning the convolution kernel $R$.  We show, using the assumptions stated for Theorem \ref{aes}, that the Fourier transform of the convolution kernel is Lipschitz continuous in $\epsilon$. 
\begin{lem}
One has  
\beq
|\hat{R}(\tilde{s};\epsilon) - \hat{R}(\tilde{s};\epsilon^{'})| \leq |\epsilon - \epsilon^{'}| |\tilde{s}|\gnorm{x \zeta(x)}_{L^{1}(\mathbb{R})}. \nonumber
\eeq
\label{convcont}
\end{lem}
\begin{proof}
With
\beq
\hat{R}(\tilde{s};\epsilon) = \int_{-\infty}^{\infty}e^{-i \tilde{s} x}R(x;\epsilon) dx, \nonumber
\eeq
we have
\beq
\hat{R}(\tilde{s};\epsilon) = \int_{-\infty}^{\infty} e^{-i \epsilon \tilde{s} x}\zeta(x) dx. \nonumber
\eeq
Using the Mean-Value Theorem, one gets
\[
|\hat{R}(\tilde{s};\epsilon) - \hat{R}(\tilde{s};\epsilon^{'})| \leq |\epsilon - \epsilon^{'}| |\tilde{s}| \gnorm{
\cdot \zeta(\cdot)}_{L^{1}(\mathbb{R})}. 
\]
Thus the result is shown.  Using Hypothesis $H4$, which amounts to assuming $\gnorm{\cdot \zeta(\cdot)}_{L^{1}(\mathbb{R})}  < \infty$, we also have that the bound is meaningful.   
\end{proof}
Let $R_{\epsilon}(\cdot) = R(\cdot,\epsilon)$.  We then show
\begin{lem}  
Given Hypotheses $H3$ and $H5$, for $f \in L^{2}(S_{T})$,  
\beq
R_{\epsilon} \ast f = R_{\epsilon} \ast \left(\sum_{j=-\infty}^{\infty} \hat{f}_{j} e_{j}(x) \right) = \sum_{j=-\infty}^{\infty} \hat{f}_{j}\hat{R_{\epsilon}}\left(\frac{2\pi j}{T}\right) e_{j}(x). \nonumber
\eeq    
\label{convcompute}
\end{lem}
\begin{proof}
First we note that 
\[
\sum_{j=-\infty}^{\infty}|\hat{f}_{j}| |R_{\epsilon} \ast e_{j}| \leq \left|\left|f \right|\right|_{2}\left(\sum_{j=-\infty}^{\infty}\hat{R_{\epsilon}}^{2}\left(\frac{2\pi j}{T}\right)\right)^{1/2} < \infty,
\]
since for $j\neq 0$ 
\begin{eqnarray}
R_{\epsilon}\ast e_{j}(x) = & \ds{\int_{-\infty}^{\infty}}R(y-x;\epsilon)\frac{e^{-i2\pi jy/T}}{\sqrt{T}}dy \nonumber \\
= & e_{j}(x) \ds{\int_{-\infty}^{\infty}} R(\tilde{y};\epsilon)e^{-i(2\pi j\tilde{y})/T} d\tilde{y}, \nonumber 
\end{eqnarray}
$R_{\epsilon}\ast e_{0}(x)=\hat{R_{\epsilon}}(0)e_{0}(x)$, and $\hat{R_{\epsilon}}(\tilde{s}) \leq (1+\epsilon|\tilde{s}|)^{-1/2-\tilde{\epsilon}}$ by Hypothesis $H5$.  By Hypothesis $H3$, $\gnorm{ R_{\epsilon}(x)}_{L^{1}(\mathbb{R})} = 1$ , and by a corollary to the Dominated Convergence Theorem (\cite{foll}, Theorem 2.25) one has
\[
R_{\epsilon} \ast \left(\sum_{j=-\infty}^{\infty} \hat{f}_{j} e_{j}(x) \right) = \sum_{j=-\infty}^{\infty} \hat{f}_{j} R_{\epsilon}\ast e_{j}(x), 
\]
or
\[
R_{\epsilon} \ast f = \sum_{j=-\infty}^{\infty} \hat{f}_{j}\hat{R_{\epsilon}}\left(\frac{2\pi j}{T}\right) e_{j}(x), 
\]
and the result is shown.  
\end{proof}
From the previous lemma, one gets
\begin{lem}
Let $s>\frac{1}{2}$.  One has that 
\beq
\opHnorm{R_{\epsilon}\ast|f|^2} \leq \opHnorm{f}^2. \nonumber
\eeq
\label{hsconv}
\end{lem}
\begin{proof}
By definition, one has  
\beq
\opHnorm{R_{\epsilon}\ast|f|^2}^{2} = \sum_{j=-\infty}^{\infty}(1+\frac{4\pi^{2}}{T^{2}}j^{2})^{s}|(R_{\epsilon}\ast|f|^{2})^{\wedge}_{j}|^{2}. \nonumber
\eeq
Let $|f|^{2} = \sum_{j} \hat{q}_{j} e_{j}(x)$, where $e_{j}(x)$ is as in \eqref{fs}, so that the Fourier series of $R_{\epsilon}\ast|f|^{2}$, using Lemma \ref{convcompute}, is 
\beq
R_{\epsilon}\ast|f|^{2}(x) = \sum_{j=-\infty}^{\infty}\hat{q}_{j} \hat{R_{\epsilon}}\left(\frac{2 \pi j}{T}\right)e_{j}(x). 
\label{confs}
\eeq
Hence, 
\beq
|(R_{\epsilon}\ast|f|^{2})^{\wedge}_{j}| = \left|\hat{q}_{j} \hat{R_{\epsilon}}\left(\frac{2\pi j}{T}\right)\right| \leq |\hat{q}_{j}| = |(|f|^{2})^{\wedge}_{j}|, \nonumber
\eeq  
and therefore  
\beq
\opHnorm{R_{\epsilon}\ast|f|^2} \leq \opHnorm{|f|^2} \leq \opHnorm{f}^{2}, \nonumber
\eeq
where the last inequality comes from the fact that $H_{s}(S_{T})$ is an algebra for $s> 1/2$ \cite{foll}. 
\end{proof}
\subsection{Local-in-Time Well Posedness of the Nonlocal Gross-Pitaevskii Equation and a Weak AES Theorem}
For the small-time argument, with initial condition $\psi_{0,\epsilon}(x)$, we rewrite \eqref{nlocgp} in the Duhamel form
\beq
\psi_{\epsilon}(x,t) = e^{-iL_{sa}t}\psi_{0,\epsilon}(x) - i\alpha \int^{t}_{0}e^{-iL_{sa}(t-t^{'})}\psi_{\epsilon}(x,t^{'})R_{\epsilon}\ast|\psi_{\epsilon}(x,t^{'})|^{2} dt^{'}, 
\label{duh}
\eeq 
where $L_{sa}= -\frac{1}{2}\p^{2}_{x} + V(x)$, with $V(x)$ assumed to be, by Hypothesis $H1$, a smooth $T$-periodic function.  As for controlling $e^{-iL_{sa}t}$, since the operator $-iL_{sa}$ is skew-adjoint, by Stone's theorem \cite{eng}, $e^{-iL_{sa}t}$ is a unitary operator from $\Lnpnorm$ to itself.  One also has that
\[
\left(e^{-iL_{sa}t} f(x) \right)^{\wedge}_{j} = e^{-i\hat{L}_{sa}(j)t}\hat{f}_{j}, 
\]
where $\hat{L}_{sa}(j)$ denotes the symbol of $L_{sa}$.  Since $L_{sa}$ is self-adjoint, $\hat{L}_{sa}(j)$ is strictly real, and so one has the bound 
\[
\opHnorm{e^{-iL_{sa}t} f} \leq \opHnorm{f}, 
\]
since $|e^{-i\hat{L}_{sa}(j)t}| = 1$ for all $j$.  

Throughout the remainder of the section, we assume $L_{sa}$ is acting on the space $H_{s}(S_{T})$ where $s>1/2$.  Using Lemma \ref{hsconv}, and defining $G_{\epsilon}(\psi)$ by,
\[
G_{\epsilon}(\psi) = e^{-iL_{sa}t}\psi_{0,\epsilon}(x) - i\alpha \int^{t}_{0}e^{-iL_{sa}(t-t^{'})}\psi(x,t^{'})R_{\epsilon}\ast|\psi(x,t^{'})|^{2} dt^{'}, 
\]
one has for $s>1/2$  
\[
\sup_{t\in[0,\tilde{T}]}\opHnorm{G_{\epsilon}(\psi)} \leq \opHnorm{\psi_{0,\epsilon}} + \tilde{T} \left(\sup_{t\in[0,\tilde{T}]}\opHnorm{\psi}\right)^{3}. 
\]
We then choose constant $C>0$ such that 
\[
\opHnorm{\psi_{0,\epsilon}} \leq C,
\]
for $\epsilon \geq 0$, with $C$ independent of $\epsilon$.  We define the metric space $B_{r}$, for $r > C$, by
\[
B_{r} = \left \{\psi(x,t) \in  L^{\infty}(H_{s}(S_{T});[0,\tilde{T}]): \sup_{t\in[0,\tilde{T}]}\opHnorm{\psi} \leq r\right\}, 
\]
where
\[
 L^{\infty}(H_{s}(S_{T});[0,\tilde{T}]) = \left\{ \psi(x,t) \in H_{s}(S_{T}) ~ \forall t\in[0,\tilde{T}] : \sup_{t\in[0,\tilde{T}]}\opHnorm{\psi} < \infty  \right\}.
\]
The map $G_{\epsilon}$ takes $B_{r}$ to $B_{r}$ for $\tilde{T}\leq \frac{r-C}{r^{3}}$.  It is straightforward to show, again using Lemma \ref{hsconv}, that $G_{\epsilon}$ is a contraction for $\tilde{T}< \frac{1}{3r^{2}}$, and therefore, using the Banach Fixed Point Theorem \cite{tao}, one has local well posedness for initial condition $\opHnorm{\psi_{0,\epsilon}(x)} \leq C$, $s>1/2$, on the space $B_{r}$ for 
\[
\tilde{T} < \min \left\{\frac{r-C}{r^{3}},\frac{1}{3r^{2}} \right\}.
\]

From the local well-posedness result, we now prove
\begin{lem}
One has that $\gnorm{\psi_{\epsilon}(\cdot,t)}_{H_{s}(S_{T})}$, $\gnorm{\psi_{\epsilon}(\cdot,t)}_{L^{\infty}(S_{T})}$, and $\gnorm{\psi^{2}_{\epsilon}(\cdot,t)}_{L^{\infty}(S_{T})}$ are continuous functions of time for $t\in[0,\tilde{T}]$.  Further, the Fourier coefficients of $\psi_{\epsilon}(x,t)$ and $\psi^{2}_{\epsilon}(x,t)$ are continuous functions of time for $t\in[0,\tilde{T}]$. 
\label{fourierlemma}
\end{lem}
\begin{proof}
Let $t,\tilde{t} \in [0,\tilde{T}]$, and $\epsilon \geq 0$.  Letting $\psi_{\epsilon}(x,t)$ denote the solution for initial condition $\psi_{0,\epsilon}(x)$, we have 
\begin{eqnarray}
\opHnorm{\psi_{\epsilon}(\cdot,t) - \psi_{\epsilon}(\cdot,\tilde{t})} & \leq &  \opHnorm{\left(e^{-iL_{sa}t}-e^{-iL_{sa}\tilde{t}} \right) \psi_{0,\epsilon}} + \nonumber \\
& & \opHnorm{\int_{\tilde{t}}^{t} e^{-iL_{sa}(t-t^{'})}\mathcal{N}(\psi_{\epsilon}) dt^{'}} \nonumber \\
& & \opHnorm{\int_{0}^{\tilde{t}}(e^{-iL_{sa}(t-t^{'})} - e^{-iL_{sa}(\tilde{t}-t^{'})})\mathcal{N}(\psi_{\epsilon}) dt^{'}} \nonumber
\end{eqnarray}
where 
\[
\mathcal{N}(\psi_{\epsilon}) = \psi_{\epsilon}(x,t^{'})R_{\epsilon}\ast|\psi_{\epsilon}(\cdot,t^{'})|^{2}(x).
\]
For the first term, we have that   
\[
\opHnorm{\left(e^{-iL_{sa}t}-e^{-iL_{sa}\tilde{t}} \right) \psi_{0,\epsilon}}^{2} \leq \sum_{j=-\infty}^{\infty}\left<j\right>^{s}\left|e^{-i\hat{L}_{sa}(j)t}-e^{-i\hat{L}_{sa}(j)\tilde{t}}\right|^{2}\left|\hat{\psi}_{0,\epsilon}(j)\right|^{2}. 
\]
Using the Dominated Convergence Theorem shows that this term then vanishes as $t\rightarrow \tilde{t}$ or vice versa.  The local well posedness result and Lemma \ref{hsconv} ensures that 
\[
\opHnorm{\mathcal{N}(\psi_{\epsilon})} \leq r^{3},
\]  
so we have 
\[
 \opHnorm{\ds{\int_{\tilde{t}}^{t} e^{-iL_{sa}(t-t^{'})}\mathcal{N}(\psi_{\epsilon})} dt^{'}} \leq r^{3}|t-\tilde{t}|.
\]
Using a dominated convergence argument shows that the remaining term must also vanish as $t\rightarrow \tilde{t}$.  Thus we have shown that $\opHnorm{\psi_{\epsilon}(\cdot,t)}$ is a continuous function in $t$ for $t\in[0,\tilde{T}]$.  By a Sobolev embedding \cite{foll}, we also have that $\gnorm{\psi_{\epsilon}(\cdot,t)}_{L^{\infty}(S_{T})}$ is continuous in $t$ and bounded above by $r$. Given that 
\[
\gnorm{\psi^{2}_{\epsilon}(\cdot,t) - \psi^{2}_{\epsilon}(\cdot,\tilde{t})}_{L^{\infty}(S_{T})} \leq 2r\gnorm{\psi_{\epsilon}(\cdot,t) - \psi_{\epsilon}(\cdot,\tilde{t})}_{L^{\infty}(S_{T})},
\]
we then see that $\gnorm{\psi^{2}_{\epsilon}(\cdot,t)}_{L^{\infty}(S_{T})}$ is continuous as well.  This result then immediately gives that the Fourier coefficients of $\psi_{\epsilon}(x,t)$ and $\psi^{2}_{\epsilon}(x,t)$ are also continuous in time since 
\begin{eqnarray}
\left|\hat{\psi}_{\epsilon}(j,t) - \hat{\psi}_{\epsilon}(j,\tilde{t})\right| &\leq & \sqrt{T} \gnorm{\psi_{\epsilon}(\cdot,t)-\psi_{\epsilon}(\cdot,\tilde{t})}_{L^{\infty}(S_{T})}, \nonumber \\
\left|\hat{\psi^{2}_{\epsilon}}(j,t) - \hat{\psi^{2}_{\epsilon}}(j,\tilde{t})\right| &\leq & \sqrt{T} \gnorm{\psi^{2}_{\epsilon}(\cdot,t)-\psi^{2}_{\epsilon}(\cdot,\tilde{t})}_{L^{\infty}(S_{T})}, \nonumber
\end{eqnarray}
for $j\neq 0$.  The $j=0$ case is treated identically, so the result is proved. 
\end{proof}

Taking $\epsilon$ and $\epsilon^{'}$ to be nonnegative, we now choose the initial conditions to be continuous in $\epsilon$ with respect to the $\opHnorm{\cdot}$-norm, {\it i.e.} 
\beq
\lim_{\epsilon \rightarrow \epsilon^{'}} \opHnorm{\psi_{0,\epsilon} - \psi_{0,\epsilon^{'}}}=0.  
\label{contofics}
\eeq
We then prove
\begin{lem}
For $\epsilon \geq 0$, $t\in[0,\tilde{T}]$, and $s>1/2$, if the initial condition $\psi_{0,\epsilon}$ is continuous in $\epsilon$ with respect to the $\opHnorm{\cdot}$-norm, the solution $\psi_{\epsilon}(x,t)$ is continuous in $\epsilon$ with respect to the $\opHnorm{\cdot}$-norm, {\it i.e.}
\[
\lim_{\epsilon \rightarrow \epsilon^{'}} \opHnorm{\psi_{\epsilon}(\cdot,t) - \psi_{\epsilon^{'}}(\cdot,t)}=0.  
\]
\label{duhcont}
\end{lem}
\begin{proof}
Choosing $\epsilon,\epsilon^{'} \geq 0$, one has that  
\beq
\ba{rl}
\opHnorm{\psi_{\epsilon^{'}}(\cdot,t)-\psi_{\epsilon}(\cdot,t)} \leq & \opHnorm{\psi_{0,\epsilon^{'}}-\psi_{0,\epsilon}}\\
&\\
& + \ds{\int_{0}^{t}}\opHnorm{\psi_{\epsilon^{'}}R_{\epsilon^{'}}\ast\left|\psi_{\epsilon^{'}}\right|^{2} - \psi_{\epsilon}R_{\epsilon}\ast\left|\psi_{\epsilon}\right|^{2}}dt^{'}.
\ea \nonumber
\eeq
In the last term of the above inequality, the convolutions depend on the different values of $\epsilon$.  Thus 
\beq
\ba{rl}
\psi_{\epsilon^{'}}R_{\epsilon^{'}}\ast\left|\psi_{\epsilon^{'}}\right|^{2} - \psi_{\epsilon}R_{\epsilon}\ast\left|\psi_{\epsilon}\right|^{2} = & (\psi_{\epsilon^{'}}-\psi_{\epsilon})R(\cdot;\epsilon^{'})\ast\left|\psi_{\epsilon^{'}}\right|^{2} \\
& \\
& + \psi_{\epsilon}(R(\cdot;\epsilon^{'})-R(\cdot;\epsilon))\ast\left|\psi_{\epsilon^{'}}\right|^{2} \\
& \\
& + \psi_{\epsilon}R(\cdot,\epsilon)\ast(\left|\psi_{\epsilon^{'}}\right|^{2}-\left|\psi_{\epsilon}\right|^{2}).
\ea \nonumber
\eeq
For $t\in[0,\tilde{T}]$, using the local-in-time well posedness result, the following inequality
\beq
\ba{rl}
\opHnorm{\psi_{\epsilon^{'}}R_{\epsilon^{'}}\ast\left|\psi_{\epsilon^{'}}\right|^{2} - \psi_{\epsilon}R_{\epsilon}\ast\left|\psi_{\epsilon}\right|^{2}} \leq & 3r^{2}\opHnorm{\psi_{\epsilon^{'}}-\psi_{\epsilon}} \\
& \\
& + r\opHnorm{(R_{\epsilon^{'}}-R_{\epsilon})\ast\left|\psi_{\epsilon^{'}}\right|^{2}} \\
\ea \nonumber
\eeq
holds.  To control the last term in this inequality, set (as in Lemma \ref{hsconv}) 
\[
\left|\psi_{\epsilon^{'}}(x,t)\right|^{2} = \sum_{j} \hat{q}_{j}(t)e_{j}(x), 
\]
from which one gets that
\[
(R(\cdot;\epsilon^{'})-R(\cdot;\epsilon))\ast\left|\psi_{\epsilon^{'}}(\cdot,t)\right|^{2} = \sum_{j}\hat{q}_{j}(t)\left(\hat{R}\left(\frac{2\pi j}{T};\epsilon^{'}\right)-\hat{R}\left(\frac{2\pi j}{T};\epsilon\right)\right)e_{j}(x). 
\] 
From the local-in-time well posedness result 
\[
\opHnorm{\psi^{2}_{\epsilon^{'}}(\cdot,t)}\leq r^{2}, 
\]
so that using Lemma \ref{convcont}, one gets the pointwise estimate
\[
\lim_{\epsilon \rightarrow \epsilon^{'}} \left|\hat{q}_{j}(t)\right|\left| \hat{R}\left(\frac{2\pi j}{T};\epsilon^{'}\right)-\hat{R}\left(\frac{2\pi j}{T};\epsilon\right)\right| = 0, 
\]
where the index $j$ is arbitrary.  Using Hypothesis $H5$, $\hat{R}(\cdot;\epsilon)$ is uniformly bounded in $j$, so the terms 
\[
\left|\hat{q}_{j}(t)\right|\left| \hat{R}\left(\frac{2\pi j}{T};\epsilon^{'}\right)-\hat{R}\left(\frac{2\pi j}{T};\epsilon\right)\right|
\]
are uniformly bounded for all $j$.  Using the Dominated Convergence Theorem, one sees that 
\beq
\lim_{\epsilon \rightarrow \epsilon^{'}}\opHnorm{(R(\cdot;\epsilon^{'})-R(\cdot;\epsilon))\ast\left|\psi_{\epsilon^{'}}(\cdot,t)\right|^{2}} = 0. 
\label{conv_continu}
\eeq
Then
\[
\opHnorm{\psi_{\epsilon^{'}}(\cdot,t)-\psi_{\epsilon}(\cdot,t)} \leq \tilde{A}(\tilde{T},\epsilon,\epsilon^{'})+ 3r^{2}\int_{0}^{t}\opHnorm{\psi_{\epsilon^{'}}(\cdot,t^{'})-\psi_{\epsilon}(\cdot,t^{'})} dt^{'},
\]
where 
\[
\tilde{A}(\tilde{T},\epsilon,\epsilon^{'}) = \opHnorm{\psi_{0,\epsilon^{'}}-\psi_{0,\epsilon}} + r\tilde{T}\sup_{t\in[0,\tilde{T}]}\tilde{C}(t;\epsilon,\epsilon^{'}), 
\]
with
\beq
\tilde{C}(t;\epsilon,\epsilon^{'}) = \opHnorm{(R(\cdot;\epsilon^{'})-R(\cdot;\epsilon))\ast\left|\psi_{\epsilon^{'}}(\cdot,t)\right|^{2}}. \nonumber 
\eeq
The term $\opHnorm{\psi_{0,\epsilon^{'}}-\psi_{0,\epsilon}}$ vanishes as $\epsilon \rightarrow \epsilon^{'}$ by assumption (see \eqref{contofics}).  Further, since, as shown in Lemma \ref{fourierlemma}, the coefficients $\hat{q}_{j}(t)$ are continuous in time, this makes the term $\tilde{C}(t;\epsilon,\epsilon^{'})$ continuous in time since it is a uniform sum of continuous functions.  Thus, the supremum is attained at some time $t^{\ast}$, and since \eqref{conv_continu} holds for any $t\in[0,\tilde{T}]$, one has that 
\[
\lim_{\epsilon \rightarrow \epsilon^{'}} \varrho(\tilde{T};\epsilon,\epsilon')= 0,
\]
where $ \varrho(\tilde{T};\epsilon,\epsilon') = \sup_{t\in[0,\tilde{T}]}\tilde{C}(t;\epsilon,\epsilon^{'})$.  Therefore, 
\[
\lim_{\epsilon \rightarrow \epsilon^{'}} \tilde{A}(\tilde{T},\epsilon,\epsilon^{'}) = 0. 
\]
We know from Lemma \ref{fourierlemma} that $\opHnorm{\psi_{\epsilon^{'}}(\cdot,t)-\psi_{\epsilon}(\cdot,t)}$ is a continuous function in $t$ for $t\in[0,\tilde{T}]$.  Using Gronwall's inequality \cite{tao}, one gets
\[
\opHnorm{\psi_{\epsilon^{'}}(\cdot,t)-\psi_{\epsilon}(\cdot,t)} \leq \tilde{A}(\tilde{T},\epsilon,\epsilon^{'})e^{3r^{2}t},
\]
and the result is therefore proved.  
\end{proof}
Since $s>1/2$, it follows that
\beq
\gnorm{\psi_{\epsilon^{'}}(\cdot,t)-\psi_{\epsilon}(\cdot,t)}_{L^{\infty}(S_{T})} \leq \opHnorm{\psi_{\epsilon^{'}}(\cdot,t)-\psi_{\epsilon}(\cdot,t)}, \nonumber
\eeq
and thus 
\beq
\gnorm{\psi_{\epsilon^{'}}(\cdot,t)-\psi_{\epsilon}(\cdot,t)}_{L^{\infty}(S_{T})} \leq \tilde{A}(\tilde{T},\epsilon,\epsilon^{'})e^{3r^{2}t} \nonumber.
\eeq
Therefore the lemma establishes that \eqref{nlocgp} is local in time AES to \eqref{locgp}.  Further, once a global-in-time well posedness result holds for \eqref{nlocgp} (which amounts to establishing a uniform bound on $r$ for all time), the above lemma immediately furnishes a global in time AES result.     

\subsection{Global-in-Time Well Posedness and the AES Theorem}

As established in \cite{decon}, \eqref{nlocgp} has at least two conserved quantities: the $L_{2}(S_{T})$ norm and the Hamiltonian
\[
\mathcal{H}(\psi) = \frac{1}{2}\int_{S_{T}}( |\psi_{x}|^{2} + 2V(x)|\psi|^{2} + \alpha |\psi|^{2}R_{\epsilon}\ast |\psi|^{2})dx. 
\]
We choose $\psi(x,\cdot) \in H_{1}(S_{T})$.  Following the argument in \cite{linar}, with $R_{\epsilon}$ positive by Hypothesis $H2$, $\alpha = 1$, and $\psi_{\epsilon}(x,t)$ a solution to \eqref{nlocgp} on time interval $t\in[0,\tilde{T}]$, with initial condition $\psi_{\epsilon,0}(x) \in H_{1}(S_{T})$, then one has that 
\[
\ba{rl}
\gnorm{\p_{x}\psi_{\epsilon}(\cdot,t)}^{2}_{L^{2}(S_{T})} & \leq 2|\mathcal{H}(\psi_{\epsilon}(x,t))| + 2\ds{\int_{S_{T}}}|V(x)||\psi_{\epsilon}(x,t)|^{2}dx \\
&\\
& \leq 2|\mathcal{H}(\psi_{\epsilon,0})|+2\gnorm{V}_{L^{\infty}(S_{T})}\gnorm{\psi_{\epsilon,0}}^{2}_{L^{2}(S_{T})}. 
\ea 
\]
Using Young's inequality, one has 
\[
\int_{S_{T}} |\psi_{\epsilon,0}|^{2}R_{\epsilon}\ast|\psi_{\epsilon,0}|^{2}dx \leq \gnorm{\psi_{\epsilon,0}}^{2}_{L^{\infty}(S_{T})} \gnorm{\psi_{\epsilon,0}}^{2}_{L^{2}(S_{T})} \leq \gnorm{\psi_{\epsilon,0}}^{4}_{H_{1}(S_{T})}. 
\]
Thus the assumptions guarantee that $|\mathcal{H}(\psi_{\epsilon,0})|<\infty$.  Since the $L^{2}(S_{T})$ norm is also conserved, there exists a constant $\tilde{M}$ such that 
\beq
\sup_{t\in[0,\tilde{T}]} \gnorm{\psi_{\epsilon}(\cdot,t)}_{H_{1}(S_{T})} \leq C_{u}.
\label{globalcontrol}
\eeq
This bound is independent of $t$.  If we now try to iterate our local-in-time well posedness argument onto a time interval $[\tilde{T}-\bar{\epsilon},\tilde{\tilde{T}})$, where $0< \bar{\epsilon}\ll1$ is chosen so that the intervals $[0,\tilde{T})$ and $[\tilde{T}-\bar{\epsilon},\tilde{\tilde{T}})$ overlap, then for the new interval we may let $C_{u}$ take the role of the value $C$. We must work on a ball $B_{\tilde{r}}$ with $\tilde{r} > C_{u}$ and 
\beq
\tilde{\tilde{T}} < \min\left\{\frac{\tilde{r}-C_{u}}{\tilde{r}^{3}},\frac{1}{3\tilde{r}^{2}} \right\}. \nonumber
\eeq
Since the inequality \eqref{globalcontrol} is independent of time, one can repeat the derivation of \eqref{globalcontrol} on the time interval $[0,\tilde{\tilde{T}})$ and obtain the same bound.  Thus one can iterate the local argument such that the value of $\tilde{r}$ need not increase, and thus the width of the new intervals can be set to a fixed value.  This establishes for the repulsive case a global existence of solutions to \eqref{nlocgp} in $H_{1}(S_{T})$ for $\epsilon \geq 0$.  As argued above, one can immediately extend the argument in Lemma \ref{duhcont} so that one has a global AES theorem. 


\section{Stability: The Linearization and Its Properties }
Having established Theorem \ref{aes}, we turn to analyzing the stability of \eqref{sol}.  Note, throughout the remainder of the paper we assume Hypotheses $H1^{'}-H4^{'}$ as listed in the Introduction.  Writing \eqref{sol} as $\psi(x,t) = \phi_{\omega}(x)e^{-i\omega t}$, and introducing the transformation $\tau = \omega t$, we see that $\phi_{\omega}$ is a stationary solution of the equation
\beq
i \psi_{\tau} = -\frac{1}{2}\p^{2}_{x}\psi + \alpha \psi \int_{-\infty}^{\infty} R(x-y;\epsilon)|\psi(y,t)|^{2}dy + V(x)\psi - \omega \psi.
\label{tgp}
\eeq 
With $\psi(x,\tau) = u(x,\tau) + i v(x,\tau)$, we rewrite \eqref{tgp} as 
\beq
\left(
\ba{c}
u \\
v
\ea
\right)_{\tau} = J \left\{
\tilde{L}_{0}
\left(
\ba{c}
u \\
v
\ea
\right)
 + \alpha \left( 
\ba{c}
u R_{\epsilon}\ast (u^2 + v^2) \\
v R_{\epsilon}\ast (u^2 + v^2)
\ea
\right) 
\right\}, \label{realandimagpart}
\eeq
where 
\[ J = \left(
\ba{cc}
0 & 1 \\
-1 & 0
\ea
\right), 
\]
\[
\tilde{L}_{0} = \left(
\ba{cc}
L_{0} & 0 \\
0 & L_{0}
\ea
\right),
\]
and $L_{0}= -\frac{1}{2}\p^{2}_{x} + V(x) - \omega$.  Note, \eqref{realandimagpart} is posed over $H^{2}_{1}(S_{2\pi/k})$, but the global-well posedness result established in the last section carries over without issue.  We set $V(x)=V_{0}\sin^{2}(kx)$ from Hypothesis $H1^{'}$.  Letting
\[
u_{\omega}(x) =\phi_{\omega,r}(x) = \sqrt{B}\cos(x), ~ v_{\omega}(x) = \phi_{\omega,i}(x)= \sqrt{B+A} \sin(x), 
\]
and equating $u = u_{\omega} + \tilde{\epsilon} w(x,\tau)$ and $v = v_{\omega} + \tilde{\epsilon} z(x,\tau)$, and collecting all $\mathcal{O}(\tilde{\epsilon})$ terms, we get the linearized system
\beq
\left(
\ba{c}
w \\
z
\ea
\right)_{\tau} = J \left\{
\tilde{L}_{0}
\left(
\ba{c}
w \\
z
\ea
\right)
+ \alpha \left( 
\ba{c}
w R_{\epsilon}\ast (u_{\omega}^2 + v_{\omega}^2) + 2 u_{\omega} R_{\epsilon} \ast (u_{\omega} w + v_{\omega} z) \\
z R_{\epsilon}\ast (u_{\omega}^2 + v_{\omega}^2) + 2 v_{\omega} R_{\epsilon} \ast (u_{\omega} w + v_{\omega} z)
\ea
\right) 
\right\}. \nonumber
\eeq

With $A = -\ds{\frac{V_{0}}{\alpha \beta(k;\epsilon)}}$, we have
\beq
R_{\epsilon}\ast (u_{\omega}^2 + v_{\omega}^2) = R_{\epsilon} \ast (B + A\sin^{2}(kx)) = B + A R_{\epsilon} \ast \left(\frac{1-\cos(2kx)}{2} \right). \nonumber
\eeq
Thus, since $R_{\epsilon}$ is an even function by Hypothesis $H2^{'}$, we write
\beq
R_{\epsilon}\ast (u_{\omega}^2 + v_{\omega}^2) = B + A\left(\frac{1-\beta(k;\epsilon)}{2} +\beta(k;\epsilon)\sin^{2}(kx)\right), \nonumber  
\eeq
and we have, for $\alpha = 1$,
\beq
L_{0} + R_{\epsilon}\ast (u_{\omega}^2 + v_{\omega}^2) = -\frac{1}{2}\left(\p^{2}_{x} + k^2 \right). \nonumber
\eeq
Introducing the transformation $x \rightarrow kx$, so that the potential and \eqref{sol} are now $2\pi$-periodic functions, and defining 
\beq
L_{c} = \left(
\ba{cc}
-\frac{k^{2}}{2}(\p^{2}_{x}+1) & 0 \\
0 & -\frac{k^{2}}{2}(\p^{2}_{x}+1),
\ea \right), \nonumber
\eeq
we rewrite the linearized system with $\alpha =1$ as 
\beq
\left(
\ba{c}
w \\
z
\ea
\right)_{\tau} = J \left\{
L_{c}
\left(
\ba{c}
w \\
z
\ea
\right)
+2 \left( 
\ba{c}
 u_{\omega} R_{k,\epsilon} \ast (u_{\omega} w + v_{\omega} z) \\
 v_{\omega} R_{k,\epsilon} \ast (u_{\omega} w + v_{\omega} z)
\ea
\right) 
\right\}, \nonumber
\eeq
where $R_{k,\epsilon}(x) = R_{k}(x;\epsilon) = \ds{\frac{1}{k}}R\left(\ds{\frac{x}{k}};\epsilon\right)$.  Defining
\beq
\bar{R}_{k,\epsilon} = \left(
\ba{cc}
R_{k,\epsilon}\ast & 0 \\
0 & R_{k,\epsilon}\ast
\ea \right), \nonumber
\eeq
and letting 
\beq
D = \sqrt{1 + \frac{A}{B}}, 
\label{defofd}
\eeq
we can rewrite the linearized system as
\beq
\left(
\ba{c}
w \\
z
\ea
\right)_{\tau} =  JL \left(
\ba{c}
w \\
z
\ea
\right), \nonumber 
\eeq
with the operator $L$ given by
\beq
L = 
L_{c}+2 B \left( 
\ba{cc}
\cos(x) & 0 \\
0 & D \sin(x) 
\ea
\right) \bar{R}_{k,\epsilon}
\left(
\ba{cc}
\cos(x) & D \sin(x) \\
\cos(x) & D \sin(x)
\ea
\right).
\nonumber
\eeq
Using separation of variables, {\it i.e.} $w(x,\tau) = w(x) e^{\lambda \tau}$ and $v(x,\tau) = v(x) e^{\lambda \tau}$, formally gives us an eigenvalue problem.  We now study the spectrum of $JL$ over $H^{2}_{2}(S_{2\pi n}) \subset L^{2}_{2}(S_{T})$.  Note, the fact we are working over the space  $H^{2}_{2}(S_{2\pi n})$ reflects the fact that we have separated the perturbations of the exact solution into real and imaginary parts.    

\subsection{The Eigenvalue Problem on $S_{2\pi n}$}
We wish to solve the spectral problem
\beq
JL \left(\ba{c}w \\ z \ea\right) = \lambda \left(\ba{c}w \\ z \ea\right), ~ w(x + 2\pi n) = w(x), ~ z(x+2\pi n) = z(x), \nonumber
\eeq
where $n \in \mathbb{N}$.  As will be shown after this section, the operator $JL$ on the domain $\mbox{D}(JL) = H^{2}_{2}(S_{2\pi n}) \subset L^{2}_{2}(S_{2\pi n})$ has a compact resolvent operator.  Therefore the spectrum,  $\sigma(JL)$, of the operator $JL$ is discrete, and solving the eigenvalue problem is sufficient to determine the spectrum.  To find the spectrum of $JL$, we note that an arbitrary $2 \pi n$-periodic function, $f(x)$, can be decomposed as
\beq
f(x) = \ds{\sum_{m} \hat{f}_{m} e^{-i\frac{m}{n} x}}= \ds{\sum_{m}\hat{f}_{m} e^{-i\frac{\tilde{m}n -  r}{n} x}} = \ds{\sum_{r=0}^{n-1}\tilde{f}_{r}(x)e^{i\frac{r}{n}x}},
\nonumber
\eeq
where $\tilde{f}_{r}(x)$ is a $2\pi$-periodic function, and $m \equiv r \mbox{mod} ~ n$.  Therefore, one can apply a similar decomposition to $w$ and $z$ so that 
\beq
\left(\ba{c}w(x) \\ z(x) \ea\right) = \sum^{n-1}_{r=0}  \left(\ba{c}w_{r}(x) \\ z_{r}(x) \ea\right)e^{i\frac{r}{n}x}. \nonumber
\eeq
One can show, for real $\mu$, that  
\beq
JL(e^{i\mu x} \cdot) = e^{i\mu x} JL_{\mu} \nonumber
\eeq
where the operator $L_{\mu}$ is given by 
\beq
L_{\mu} =  \left\{L_{c,\mu}+2 B \left( \ba{cc} \cos(x) & 0 \\ 0 & D \sin(x) \ea \right) \bar{R}_{k,\epsilon,\mu} \left( \ba{cc}
\cos(x) & D \sin(x) \\ \cos(x) & D \sin(x) \ea \right) \right\}, \nonumber
\eeq
with 
\beq
L_{c,\mu} = \left( \ba{cc} -\frac{k^{2}}{2}((\p_{x}+ i\mu)^{2}+1) & 0 \\ 0 & -\frac{k^{2}}{2}((\p_{x}+i\mu)^{2}+1),\ea \right), \nonumber
\eeq
and
\beq
\bar{R}_{k,\epsilon,\mu} = \left(\ba{cc} R_{k,\epsilon,\mu}\ast & 0 \\ 0 & R_{k,\epsilon,\mu}\ast \ea
\right). \nonumber
\eeq
Here
\beq
R_{k,\epsilon,\mu}(x) = R_{k,\mu}(x;\epsilon) = \frac{1}{k}R\left(\frac{x}{k};\epsilon\right)e^{-i\mu x}. \nonumber
\eeq 
Thus one has for any eigenvalue $\lambda$ that 
\beq
(JL-\lambda)\left(\ba{c}w(x) \\ z(x) \ea\right) = \sum_{r=0}^{n-1}e^{i\frac{r}{n}x}\left(JL_{\frac{r}{n}}-\lambda\right) \left(\ba{c}w_{r}(x) \\ z_{r}(x) \ea\right) = 0. \nonumber
\eeq
The term $(JL_{\frac{r}{n}}-\lambda) \left(\ba{c}w_{r}(x) \\ z_{r}(x) \ea\right)$ is a $2\pi$ periodic function.  Since none of the functions share a common period shorter than $2\pi n$, the equality
\beq
\left(JL_{\frac{r}{n}}-\lambda\right) \left(\ba{c}w_{r}(x) \\ z_{r}(x) \ea\right) = 0\nonumber
\eeq  
must hold for each value of $r$.  This shows that one can decompose the spectrum of $JL$ on $H^{2}_{2}(S_{2\pi n})$ as a union of the spectra of the operators $JL_{\frac{r}{n}}$ posed on $H^{2}_{2}(S_{2\pi})$, {\it i.e.} one can write
\beq
\sigma(JL) = \bigcup^{n-1}_{r=0} \sigma(JL_{r/n}). \nonumber
\eeq  
Note, one cannot rely on standard Floquet theory since the spectral problem is not an ordinary differential equation.  In the succeeding sections we study the problem $JL_{\mu}$ on $S_{2\pi}$, with $\mu \in [0,1)$, in order to deal with arbitrary values of $r/n$.    
\subsection{Basic Results about the Convolution Kernel and Linearization}
We prove a number of technical lemmas concerning the convolution and linearization that are used throughout the remainder of the paper. 
\begin{lem}
Given Hypotheses $H2^{'}$ and $H4^{'}$, the operator $\bar{R}_{k,\epsilon,\mu}:L^{2}_{2}(S_{2\pi})\rightarrow L^{2}_{2}(S_{2\pi})$ is compact.  Further,  $\bar{R}_{k,\epsilon,\mu}$ is continuous in $\mu$, $\mu\in[0,1]$, with respect to the $\gnorm{\cdot}_{2,v}$-norm.
\label{propsaboutr}
\end{lem}
\begin{proof}
Using the same arguments as in Lemma \ref{convcompute}, one finds the Fourier series representation of $\bar{R}_{k,\epsilon,\mu}$, which we denote as $\hat{\bar{R}}_{k,\epsilon,\mu}$, as 
\beq
\hat{\bar{R}}_{k,\epsilon,\mu} = \left(
\ba{cc}
\Lambda_{k,\epsilon,\mu} & 0 \\
0 & \Lambda_{k,\epsilon,\mu}
\ea
\right), \nonumber
\eeq    
where $\Lambda_{k,\epsilon,\mu}$ is diagonal and $\left(\Lambda_{k,\epsilon,\mu} \right)_{jj} = \hat{R}_{k,\epsilon}(j-\mu)$.  Since in Hypothesis $H2^{'}$ we assume $R_{k,\epsilon} \in L^{1}\left(\mathbb{R}\right)$, by the Riemann-Lebesgue lemma \cite{foll}, we have 
\beq
\lim_{|j| \rightarrow \infty}\left(\Lambda_{k,\epsilon,\mu} \right)_{jj} = 0. \nonumber
\eeq 
Defining
\beq
R^{N}_{k,\epsilon,\mu} f = \sum_{j=-N}^{N} \hat{f}_{j}\hat{R}_{k,\epsilon}(j-\mu) e_{j}(x), \nonumber
\eeq 
we see that for $\gnorm{f}_{L^{2}(S_{2\pi})}=1$
\beq
\norm{R_{k,\epsilon,\mu} \ast f - R^{N}_{k,\epsilon,\mu}f} \leq \left(\sum_{|j|>N} |\hat{R}_{k,\epsilon}(j-\mu)|^{2}\right)^{\frac{1}{2}}, \nonumber
\eeq
Since we assume in Hypothesis $H4^{'}$ that $|\hat{\zeta}|\leq (1+|\tilde{s}|)^{-1/2-\tilde{\epsilon}}$, the above sum decays to zero as $N \rightarrow \infty$, and the operator $R_{k,\epsilon,\mu}\ast$ is a uniform limit of finite rank operators.  Therefore, so is $\bar{R}_{k,\epsilon,\mu}$, and $\bar{R}_{k,\epsilon,\mu}$ must then be compact. 

To prove the last part of the lemma, we note that for $\mu,\mu^{'} \in [0,1]$
\[
\hat{R}_{k,\epsilon}(j-\mu) - \hat{R}_{k,\epsilon}(j-\mu^{'}) = \int_{-\infty}^{\infty} R_{k}(x;\epsilon)e^{-ijx}\left(e^{i\mu x} -e^{i\mu^{'}x}\right)dx,
\]
so using Hypothesis $H2^{'}$ and the Dominated Convergence Theorem shows that $\hat{R}_{k,\epsilon}(j-\mu) \rightarrow \hat{R}_{k,\epsilon}(j-\mu^{'})$ as $\mu \rightarrow \mu^{'}$, or $\hat{R}_{k,\epsilon}(j-\mu)$ is continuous in $\mu$.   Likewise, we have, using Hypothesis $H4^{'}$,    
\[
\norm{R_{k,\epsilon,\mu}\ast - R_{k,\epsilon,\mu^{'}}\ast}^{2} \leq 2\tilde{S}(\mu) + 2\tilde{S}(\mu^{'})
\]
where 
\[
\tilde{S}(\mu) = \sum_{j=-\infty}^{\infty} \frac{1}{\left(1+|k\epsilon(j-\mu)| \right)^{1+2\tilde{\epsilon}}}.
\]
For $\mu\in[0,1]$, one has 
\[
\tilde{S}(\mu) \leq \sum_{j=-\infty}^{0} \frac{1}{\left(1+|k\epsilon j| \right)^{1+2\tilde{\epsilon}}} + \sum_{j=1}^{\infty} \frac{1}{\left(1+|k\epsilon (j-1)| \right)^{1+2\tilde{\epsilon}}},
\]
so using a dominated convergence argument, one gets $\norm{R_{k,\epsilon,\mu}\ast - R_{k,\epsilon,\mu^{'}}\ast}\rightarrow 0$ as $\mu\rightarrow \mu^{'}$. Thus $\gnorm{\bar{R}_{k,\epsilon,\mu}-\bar{R}_{k,\epsilon,\mu^{'}}}_{2,v}\rightarrow 0$ as $\mu\rightarrow \mu^{'}$, so $\bar{R}_{k,\epsilon,\mu}$  is continuous in $\mu$.     
\end{proof}
From the previous lemma one gets
\begin{lem}
The operator $L_{\mu}$ has a compact resolvent on $L^{2}_{2}(S_{2\pi})$.  
\label{compresolv}
\end{lem}
\begin{proof}
Given that $L_{\mu} = L_{c,\mu} +2 K(\epsilon;\mu;D)$, where 
\beq
K(\epsilon;\mu;D) = B  \left( \ba{cc} \cos(x) & 0 \\ 0 & D \sin(x) \ea \right) \bar{R}_{k,\epsilon,\mu} \left( \ba{cc}
\cos(x) & D \sin(x) \\ \cos(x) & D \sin(x) \ea \right), \nonumber
\eeq
one has that $K$, where we have suppressed the dependence on $\epsilon$ and $D$, is compact since it is the product of bounded and compact operators.  Further, a straightforward application of Fourier series shows that $L_{c,\mu}$ has compact resolvent on $L^{2}(S_{2\pi})$.  Let $\lambda$ be a complex number with nonzero imaginary part.  Then  
\beq
I+\left(L_{c,\mu}-\lambda\right)^{-1}K = \left(L_{c,\mu}-\lambda \right)^{-1}\left(L_{\mu}-\lambda\right). 
\label{identrelat}
\eeq        
The operator $I+\left(L_{c,\mu}-\lambda\right)^{-1}K$ is Fredholm since $\left(L_{c,\mu}-\lambda\right)^{-1}K$ is compact.  The right-hand side of \eqref{identrelat} has a trivial kernel since $L_{\mu}$ is self-adjoint.  Thus the left-hand side of \eqref{identrelat} also has a trivial kernel, which implies that $I+\left(L_{c,\mu}-\lambda\right)^{-1}K$ has a bounded inverse \cite{lax}.  Therefore, from 
\beq
\left(L_{\mu}-\lambda\right)^{-1} =\left( I+\left(L_{c,\mu}-\lambda\right)^{-1}K\right)^{-1} \left(L_{c,\mu}-\lambda \right)^{-1}, \nonumber
\eeq  
one sees that $\left(L_{\mu}-\lambda\right)^{-1}$ is the product of a bounded and a compact operator, and is therefore itself compact.  
\end{proof}
\noindent Assuming that $\lambda$ is in the resolvent of $JL_{\mu}$, and using that 
\beq
\left(JL_{\mu} - \lambda \right)^{-1} = -\left(L_{\mu}-\gamma \right)^{-1}J\left(I-(\gamma J-\lambda)\left(L_{\mu}-\gamma \right)^{-1}J\right)^{-1}, \nonumber
\eeq
where $\gamma$ is in the resolvent of $L_{\mu}$, one sees that $JL_{\mu}$ has a compact resolvent on $L^{2}_{2}(S_{2\pi})$ since $\left(JL_{\mu} - \lambda \right)^{-1}$ is the product of compact and bounded operators.  

We now need to establish some limiting behavior of the operator $\bar{R}_{k,\epsilon,\mu}$ as the nonlocality parameter $\epsilon$ becomes large.  We prove:
\begin{lem}  
Given Hypothesis $H4^{'}$, for $\mu \neq 0$, $\lim_{\epsilon \rightarrow \infty}  \norm{R_{k,\epsilon,\mu}\ast} = 0.$
\label{decayofrs}
\end{lem}
\begin{proof}
One has
\beq
\hat{R}_{k,\epsilon}(j-\mu) = \ds{\int_{-\infty}^{\infty}} R_{k}(x;\epsilon)e^{-i(j-\mu)x} dx = \ds{\int_{-\infty}^{\infty}} \zeta(x) e^{-ik\epsilon(j-\mu)x} dx = \hat{\zeta}(k\epsilon(j-\mu)). \nonumber
\eeq  
Examining the $L^{2}(S_{2\pi})$ norm of the operator $R_{k,\epsilon,\mu}\ast$, one gets
\beq
\norm{R_{k,\epsilon,\mu}\ast}^{2} \leq \sum_{j=-\infty}^{\infty} \left|\hat{\zeta}(k\epsilon(j-\mu)) \right|^{2}. \nonumber
\eeq
Note, the sum is convergent by Hypothesis $H4^{'}$.  For a given value of the nonlocality parameter $\epsilon$ and an arbitrarily chosen value of $\delta$, choose $\tilde{N}$ such that 
\beq
\sum_{|j|>\tilde{N}} \left|\hat{\zeta}(k\epsilon(j-\mu)) \right|^{2} <  \frac{\delta}{2}. \nonumber
\eeq
Next, choose $\epsilon$ large enough such that 
\beq
\sum_{j=-\tilde{N}}^{\tilde{N}} \left|\hat{\zeta}(k\epsilon(j-\mu)) \right|^{2} < \frac{\delta}{2}. \nonumber
\eeq
The second assumption does not alter the first since choosing a large $\epsilon$ value corresponds to choosing a larger value of $\tilde{N}$.  Thus, for $\mu \neq 0$,
\beq
\lim_{\epsilon \rightarrow \infty}  \norm{R_{k,\epsilon,\mu}\ast} = 0, \nonumber
\eeq
and $\bar{R}_{k,\epsilon,\mu} \rightarrow 0$ uniformly in norm as $\epsilon \rightarrow \infty$. 
\end{proof}

We finally prove that the resolvents of $JL_{\mu}$ and $JL_{0}$ converge in the $L^{2}_{2}(S_{2\pi})$-norm.  This is used to show, in effect, that the spectra of one operator is a perturbation in $\mu$ of the other. 
\begin{lem}
Suppose there exists $\mu^{\ast}\in (0,1)$ such that $\lambda$ is in the resolvent of $JL_{\mu}$ for $0\leq\mu<\mu^{\ast}$.  Further suppose that $(JL_{c,0}-\lambda)^{-1}$ exists.  Then $(JL_{\mu}-\lambda)^{-1}$ converges to $(JL_{0}-\lambda)^{-1}$ in the $L^{2}_{2}(S_{2\pi})$-norm as $\mu\rightarrow 0^{+}$. 
\label{convresolv}
\end{lem}
\begin{proof}
Define the operator $\tilde{R}_{\mu}(\lambda) = (JL_{\mu}-\lambda)^{-1}$.  Then we have that 
\[
\gnorm{\tilde{R}_{\mu}(\lambda) - \tilde{R}_{0}(\lambda)}_{2,v} \leq \gnorm{\tilde{R}_{0}(\lambda)}_{2,v}\gnorm{I - \left(J(L_{\mu}-L_{0})\tilde{R}_{0}(\lambda)+I \right)^{-1}}_{2,v}
\]
We have that 
\[
L_{\mu} - L_{0} = \tilde{L}_{c}(\mu) + \tilde{V}_{1} \left(\bar{R}_{k,\epsilon,\mu}-\bar{R}_{k,\epsilon,0} \right)\tilde{V}_{2},
\]
where $\tilde{L}_{c}(\mu) = L_{c,\mu}-L_{c,0}$, and $\tilde{V}_{1}$ and $\tilde{V}_{2}$ are such that $2K(\epsilon;\mu;D) = \tilde{V}_{1} \bar{R}_{k,\epsilon,\mu}\tilde{V}_{2}$.  Using the fact that $\tilde{V}_{1}$ and $\tilde{V}_{2}$ are bounded in $L^{2}_{2}(S_{2\pi})$ and Lemma \ref{propsaboutr},  
\[
\lim_{\mu \rightarrow 0^{+}} \gnorm{ \tilde{V}_{1} \left(\bar{R}_{k,\epsilon,\mu}-\bar{R}_{k,\epsilon,0} \right)\tilde{V}_{2}}_{2,v} = 0.
\]
Defining $\tilde{R}_{0,c}(\lambda)=(JL_{c,0}-\lambda)^{-1}$, we rewrite $\tilde{R}_{0}(\lambda)$ so that 
\[
\tilde{R}_{0}(\lambda) = \tilde{R}_{0,c}(\lambda)\left(I + 2JK(\epsilon;0;D)\tilde{R}_{0,c}(\lambda)\right)^{-1}.
\]
We then get that 
\[
\gnorm{J\tilde{L}_{c}(\mu)\tilde{R}_{0}(\lambda)}_{2,v}\leq \gnorm{\tilde{L}_{c}(\mu) \tilde{R}_{0,c}(\lambda)}_{2,v}\gnorm{\left(I + 2JK(0)\tilde{R}_{0,c}(\lambda)\right)^{-1}}_{2,v},
\]  
where $K(0)=K(\epsilon;0;D)$.  The operator $\tilde{L}_{c}(\mu)\tilde{R}_{0,c}(\lambda) = \tilde{L}_{c}(\mu)  \left(JL_{c,0} - \lambda \right)^{-1}$ is a constant coefficient operator.  Thus, using the Fourier transform, it is straightforward to show it is bounded and must vanish in the $\gnorm{\cdot}_{2,v}$-norm as $\mu \rightarrow 0^{+}$.  Thus we have that 
\[
\lim_{\mu\rightarrow 0^{+}} \gnorm{J(L_{\mu}-L_{0})\tilde{R}_{0}(\lambda)}_{2,v}=0.
\]
Taking $\mu$ sufficiently small so that $\gnorm{J(L_{\mu}-L_{0})\tilde{R}_{0}(\lambda)}_{2,v} < 1$, we have that 
\[
\gnorm{I - \left(J(L_{\mu}-L_{0})\tilde{R}_{0}(\lambda)+I \right)^{-1}}_{2,v} \leq \frac{\gnorm{J(L_{\mu}-L_{0})\tilde{R}_{0}(\lambda) }_{2,v}}{1-\gnorm{J(L_{\mu}-L_{0})\tilde{R}_{0}(\lambda) }_{2,v}},
\]
which shows that 
\[
\lim_{\mu \rightarrow 0^{+}}\gnorm{\tilde{R}_{\mu}(\lambda) - \tilde{R}_{0}(\lambda)}_{2,v} = 0.
\]
\end{proof}

\section{Stability for Small Potential and Large Offset Size}

\subsection{Computation of the Spectrum with $V_{0}=0$}  

In this section, we compute the spectrum of $JL$ over $H^{2}_{2}(S_{2\pi n})$, with $V_{0}=0$ or $D=1$ (see \eqref{defofd}).  As explained earlier, this is done by computing the spectrum of the operators $JL_{r/n}$ over $H^{2}_{2}(S_{2\pi})$, $r\in\left\{0,\cdots,n-1 \right\}$.  To do this, we notice that we can treat $JL_{\mu}$ as a constant coefficient operator with a compact perturbation.  For the remainder of the section, we assume $\mu \neq 0$ so that the compact perturbation decays uniformly to zero as $\epsilon \rightarrow \infty$.  The $\mu=0$ case is covered by noting that $J(L_{\mu}-L_{0})$ is a relatively compact perturbation of $JL_{0}$, which we note was used to prove Lemma \ref{convresolv}.  Therefore one can find the eigenvalues of $JL_{0}$ by taking limits of the eigenvalues of $JL_{\mu}$.  

Using the Fourier transform, we compute the spectrum and eigenfunctions of $JL_{c,\mu}$ explicitly.  One has
\beq
\sigma(JL_{c,\mu}) = \left\{\pm  \frac{i}{2}k^2 \left|(n-\mu)^2-1\right|: n \in \mathbb{Z} \right\}, \nonumber
\eeq  
and for $n \neq 0, 1$, the corresponding eigenfunctions for the eigenvalues on the positive imaginary axis are 
\beq 
\left(
\ba{c}
1 \\
i
\ea
\right) e^{-inx},
\label{pkr}
\eeq
while for $n = 0~ \mbox{or} ~1$,
\beq
\left(
\ba{c}
1 \\
-i
\ea
\right) e^{-inx}.
\label{nkr}
\eeq
Taking conjugates and letting $x \rightarrow -x$ gives the corresponding eigenfunctions for the eigenvalues on the negative imaginary axis.  

The eigenvalue problem for the operator $JL_{\mu}$ and eigenvalue $\lambda_{n}$ is of course to find nontrivial $\varphi_{n}\in H^{2}_{2}(S_{2\pi})$ such that
\beq
(JL_{\mu} - \lambda_{n})\varphi_{n} = 0. \nonumber
\eeq  
We write, as in Lemma \ref{compresolv}, 
\beq
(L_{c,\mu} + 2K(\epsilon;1) + \lambda_{n} J)\varphi_{n} = 0, \nonumber
\eeq
and let $\lambda_{n} = \lambda_{\infty}(n) + \lambda_{p}(n)$, where $\lambda_{\infty}(n)$ is an eigenvalue of $\sigma(JL_{c,\mu})$, and $\lambda_{p}(n)$ is a perturbation of $\lambda_{\infty}(n)$ that will be determined exactly.  We have
\beq
(L_{c,\mu}+\lambda_{\infty}(n)J + 2K(\epsilon;1) + \lambda_{p}(n) J)\varphi_{n} = 0.
\label{eval}
\eeq 
Let 
\beq
L^{\infty}_{c,\mu,n} = L_{c,\mu}+\lambda_{\infty}(n)J, \nonumber
\eeq
and
\beq
T(\epsilon;n) = 2K(\epsilon;1) + \lambda_{p}(n) J. \nonumber
\eeq
Define $P_{n}$ to be the projection onto the null space of $L^{\infty}_{c,\mu,n}$.  Since $L^{\infty}_{c,\mu,n}$ is self adjoint, we use a Lyupanov-Schmidt reduction \cite{hale} to rewrite \eqref{eval} as 
\begin{eqnarray}
\xi_{n} + M(n) (\phi_{n} + \xi_{n}) & = 0 \label{lsr1}\\
P_{n} T(\epsilon;n) (\phi_{n} + \xi_{n}) & = 0 \label{lsr2}
\end{eqnarray}
where $\varphi_{n} = \phi_{n} + \xi_{n}$, $\phi_{n}$ is in the null space of $L^{\infty}_{c,\mu,n}$, and 
\beq
M(n) = (L^{\infty}_{c,\mu,n})^{-1}(I-P_{n})T(\epsilon;n). \nonumber
\eeq
At this point, the equations \eqref{lsr1} and \eqref{lsr2} are the same as the original eigenvalue problem.  No added assumptions or constraints have been made.  Therefore solving \eqref{lsr1} and \eqref{lsr2} is equivalent to solving the original eigenvalue problem.  

Rewriting \eqref{lsr1} as 
\[
(I+M(n))\xi_{n} = -M(n)\phi_{n}, 
\]
we may formally write
\[
\xi_{n} = -(I+M(n))^{-1}M(n) \phi_{n} = (-M(n) + M^{2}(n) - \cdots)\phi_{n}. 
\]
Though this expansion is valid for sufficiently large $\epsilon$ (see Lemma \ref{decayofrs}), it is more important as a motivation to look at the terms $M^{k}(n)\phi_{n}$.  For example, let $n=0$ or $1$, with $\lambda_{\infty}(n)$ on the positive imaginary axis, so that $\phi_{n}$ is given by \eqref{nkr}.  For $D=1$, 
\[
T(\epsilon;n) = 
2 B \left( 
\ba{cc}
\cos(x) & 0 \\
0 & \sin(x) 
\ea
\right) \bar{R}_{k,\epsilon,\mu} \left(
\ba{cc}
\cos(x) & \sin(x) \\
\cos(x) & \sin(x)
\ea
\right) + \lambda_{p}(n)J, 
\]  
so that
\[
T(\epsilon;n) \phi_{n} = (B\hat{r}_{n+1}-i\lambda_{p}(n))\phi_{n} + B\hat{r}_{n+1}\tilde{\phi}_{n}, 
\]
where
\[
\tilde{\phi}_{n} =
\left( 
\ba{c}
1 \\
i
\ea
\right)e^{-i(n+2)x}, 
\]
and 
\[
\hat{r}_{n} = \hat{R}_{k,\epsilon}(n-\mu).  
\]
Note, we suppress the parameters $\epsilon$ and $\mu$ in $\hat{r}_{n}$ for the sake of clarity in the presentation.  Thus $(I-P_{n})T(\epsilon;n)\phi_{n}=B\hat{r}_{n+1}\tilde{\phi}_{n}$, and 
\[
L^{\infty}_{c,\mu,n} \tilde{\phi}_{n} = \left(\frac{k^{2}}{2}((n+2-\mu)^2-1) +i\lambda_{\infty}(n) \right)\tilde{\phi}_{n}. 
\]
Hence, for $n=0$,
\[
M(0)\phi_{0} = \frac{B\hat{r}_{1}}{k^{2}(\mu-1)^{2}}\tilde{\phi}_{0}, 
\]
and for $n=1$,
\[
M(1)\phi_{1} = \frac{B\hat{r}_{2}}{k^{2}(\mu-2)^{2}}\tilde{\phi}_{1}. 
\]

We consider $M(n)\tilde{\phi}_{n}$ for $n = 0$ or $1$.  We see that
\[
T(\epsilon;n)\tilde{\phi}_{n} = \hat{r}_{n+1}B\phi_{n} + (\hat{r}_{n+1}B+i\lambda_{p}(n))\tilde{\phi}_{n}, 
\]
and hence for $n=0$, we obtain
\[
M(0)\tilde{\phi}_{0} = \frac{\hat{r}_{1}B+i\lambda_{p}(0)}{k^{2}(\mu-1)^{2}}\tilde{\phi}_{0}. 
\]
For $n=1$, we have
\[
M(1)\tilde{\phi}_{1} = \frac{\hat{r}_{2}B+i\lambda_{p}(1)}{k^{2}(\mu-2)^{2}}\tilde{\phi}_{1}. 
\]
Define the constants $\gamma_{n}$ and $\delta_{n}$ such that $M(n)\phi_{n}=\gamma_{n}\tilde{\phi}_{n}$ and $M(n)\tilde{\phi}_{n}=\delta_{n}\tilde{\phi}_{n}$.  Therefore, equating 
\[
\xi_{n} = -\frac{\gamma_{n}}{1+\delta_{n}} \tilde{\phi}_{n} 
\]
gives a solution of \eqref{lsr1}.    

Using \eqref{lsr2}, one obtains
\[
\left<T(\epsilon;n)\phi_{n},\phi_{n}\right> -\frac{\gamma_{n}}{1+\delta_{n}} \left<T(\epsilon;n)\tilde{\phi}_{n},\phi_{n} \right>= 0. 
\]
From the work above, since $\norm{\phi_{n}}>0$, we see this reduces to
\[
(B\hat{r}_{n+1}-i\lambda_{p}(n)) - \frac{B\hat{r}_{n+1}\gamma_{n}}{1+\delta_{n}} = 0.
\]
Writing 
\[
\delta_{n} = \ds{\frac{B\hat{r}_{n+1}+i\lambda_{p}(n)}{c_{n}(\mu)}}, 
\]
where $c_{0}=k^{2}(\mu-1)^{2}$ and $c_{1}=k^{2}(\mu-2)^{2}$, we see that we want to solve the quadratic equation 
\beq
\lambda^{2}_{p}(n) -ic_{n}(\mu)\lambda_{p}(n) +c_{n}(\mu)B\hat{r}_{n+1}= 0.
\label{quad}
\eeq 
Let $\lambda_{p}(n)=\lambda^{(r)}_{p}(n)+i\lambda^{(i)}_{p}(n)$, where $\lambda^{(r)}_{p}(n)$ and $\lambda^{(i)}_{p}(n)$ are real values.  Therefore, by separating into real and imaginary parts, \eqref{quad} becomes
\begin{eqnarray}
\lambda^{(r)}_{p}(2\lambda^{(i)}_{p}-c_{n}(\mu)) & = 0 \nonumber \\
(\lambda^{(r)}_{p})^2 - (\lambda^{(i)}_{p})^2 +c_{n}(\mu)\lambda^{(i)}_{p} + c_{n}(\mu)B\hat{r}_{n+1} & = 0. \nonumber
\end{eqnarray}
If we assume $\lambda^{(r)}_{p}(n)\neq 0$, we have $\lambda^{(i)}_{p}(n) = \ds{\frac{c_{n}}{2}}$, which implies
\[
(\lambda^{(r)}_{p})^2 = -\frac{c^{2}_{n}}{4} - c_{n}(\mu)B\hat{r}_{n+1}.
\]
The right-hand side of the above expression is always negative by construction, since $\hat{r}_{n+1}>0$.  Thus $\lambda^{(r)}_{p}(n)=0$, so that $\lambda_{p}(n)=i\lambda^{(i)}_{p}(n)$.  Therefore we have
\[
\lambda_{p}(n) = \frac{i}{2}\left(c_{n}(\mu) - \left(c^{2}_{n}(\mu) + 4c_{n}(\mu)B\hat{r}_{n+1} \right)^{1/2}\right). 
\]
Since we know, as shown in Lemma \ref{decayofrs}, that $\hat{r}_{n} \rightarrow 0$ as $\epsilon \rightarrow \infty$, and we need $\lambda_{p}(n)\rightarrow 0$ as $\epsilon \rightarrow \infty$, this determines the correct sign when solving the quadratic equation in $\lambda^{(i)}_{p}(n)$.  With this choice of $\lambda_{p}(n)$, we see $\delta_{n} > 0$, so $1+\delta_{n} > 0$, and the choice of $\xi_{n}$ is well defined.  

For the case that $n \neq 0$ or $1$, proceeding in a fashion identical to that above, one shows that
\[
M(n) \phi_{n} = \gamma_{n} \tilde{\phi}_{n}, 
\]
and
\[
M(n) \tilde{\phi}_{n} = \delta_{n} \tilde{\phi}_{n}. 
\]
Here 
\beq
\tilde{\phi}_{n} = 
\left( 
\ba{c}
1 \\
-i
\ea
\right)e^{-i(n-2)x}, \nonumber
\eeq 

\beq
\gamma_{n} = \frac{B\hat{r}_{n-1}}{c_{n}(\mu)}, \nonumber
\eeq
and
\beq
\delta_{n} = \frac{B\hat{r}_{n-1}-i\lambda_{p}(n)}{c_{n}(\mu)}, \nonumber
\eeq
where $c_{n}(\mu)=k^2(n-\mu-1)^2$.

We again equate $\xi_{n} = -\ds{\frac{\gamma_{n}}{1+\delta_{n}}}\tilde{\phi}_{n}$, which solves \eqref{lsr1}, and from \eqref{lsr2} one gets a characteristic equation for $\lambda_{p}(n)$ which is
\[
\lambda_{p}^{2} + ic_{n}(\mu)\lambda_{p} + c_{n}(\mu)B\hat{r}_{n-1} = 0. 
\]  
Finally, 
\[
\lambda_{p}(n) = \frac{i}{2}\left(-c_{n}(\mu) + \left(c^{2}_{n}(\mu) + 4c_{n}(\mu)B\hat{r}_{n-1} \right)^{1/2} \right).
\]

Given that the operator $2K(\epsilon;\mu;1)$ is not symmetric with respect to conjugation followed by equating $x$ to $-x$, we must repeat the above computations except now with the expansions around the eigenvalues along the negative imaginary axis.  The process is identical to that above, and we only list the results.  For $n = 0$ or $1$, we have 
\[
\lambda_{p}(n) = \frac{i}{2} \left(-c_{n}(\mu) + \left(c^{2}_{n}(\mu) + 4c_{n}(\mu)B\hat{r}_{n-1} \right)^{1/2} \right), 
\] 
where $c_{n}(\mu)$ is given by
\beq
c_{n}(\mu) = \left\{
\ba{cc}
k^2(\mu+1)^2, & n=0, \\
\\
k^2\mu^2, & n=1.
\ea
\right. \nonumber
\eeq
The corresponding eigenfunctions are
\beq
\varphi_{n} =
\left( 
\ba{c}
1 \\
i
\ea
\right)e^{-inx} - \frac{\gamma_{n}}{1+\delta_{n}} 
\left(
\ba{c}
1 \\
-i
\ea
\right)e^{-i(n-2)x}, \nonumber
\eeq
where $\gamma_{n} = \ds{\frac{B\hat{r}_{n-1}}{c_{n}(\mu)}}$ and $\delta_{n}=\ds{\frac{B\hat{r}_{n-1}-i\lambda_{p}}{c_{n}(\mu)}}$.

Likewise, for $n\neq 0, 1$, we have 
\beq
\lambda_{p}(n) = \frac{i}{2} \left(c_{n}(\mu) - \left(c^{2}_{n}(\mu) + 4c_{n}(\mu)B\hat{r}_{n+1} \right)^{1/2} \right), \nonumber
\eeq
with $c_{n}(\mu) = k^2(n-\mu+1)^2$.  The corresponding eigenfunctions are
\beq
\varphi_{n} =
\left( 
\ba{c}
1 \\
-i
\ea
\right)e^{-inx} - \frac{\gamma_{n}}{1+\delta_{n}} 
\left(
\ba{c}
1 \\
i
\ea
\right)e^{-i(n+2)x}, \nonumber
\eeq  
with $\gamma_{n} = \ds{\frac{B\hat{r}_{n+1}}{c_{n}(\mu)}}$ and $\delta_{n}=\ds{\frac{B\hat{r}_{n+1}+i\lambda_{p}}{c_{n}(\mu)}}$.

The only issue remaining is whether we have captured the entire spectrum of $JL_{\mu}$ for each value of $\epsilon \in [0,\infty)$.  However, every eigenvalue is a perturbation of an eigenvalue in the constant coefficient case, which has only simple eigenvalues in its spectrum since $JL_{c,\mu}$ is a skew-adjoint operator with compact resolvent.  Hence we have not missed any eigenvalues due to multiplicity.  Thus we have computed $\sigma(JL_{\mu})$ for $V_{0}=0$.  
\subsection{Krein Signature}
For a purely imaginary semisimple eigenvalue $\lambda \in \sigma(JL_{\mu})$ with eigenvector $\varphi$, the Krein signature of $\lambda$ is defined as $\mbox{sgn}(\left<L_{\mu}\varphi,\varphi\right>)$ \cite{bjorn}.  Let
\beq
\alpha_{n}(B,\mu,\epsilon,k) = \frac{\gamma_{n}}{1+\delta_{n}}, \nonumber
\eeq  
which, for the eigenvalues that represent perturbations of eigenvalues on the positive imaginary axis, is given by
\beq
\alpha_{n} = \frac{B\hat{r}_{n\pm 1}}{\frac{c_{n}}{2}+\frac{\sqrt{c^{2}_{n} + 4c_{n}B\hat{r}_{n\pm1}}}{2}+B\hat{r}_{n\pm1}}, \nonumber
\eeq
where the $+$ in $\pm$ corresponds to choosing $n=0,1$ and the $-$ corresponds to $n \neq 0,1$.  A similar expression can be derived for the eigenvalues starting on the negative imaginary axis.  

Given the definition for $\alpha_{n}$, it is straightforward to show for $D=1$ and $\varphi_{n} = \phi_{n}+\xi_{n}$, that 
\begin{eqnarray}
\left<L_{\mu}\varphi_{n},\varphi_{n}\right> & = & 2\pi k^2\left(\left(n-\mu \right)^2-1 + \alpha_{n}^2\left(\left(n\pm2-\mu\right)^2-1\right) \right) \nonumber \\
& &+ 4\pi B\hat{r}_{n\pm1}(1-\alpha_{n})^2.\label{ks}
\end{eqnarray}

Along the positive imaginary axis, one again lets the $-$ of $\pm$ correspond to the case $n \neq 0, 1$, while we take $+$ for $n=0, 1$.  This relationship is reversed on the negative imaginary axis.  Thus we see, starting on the positive imaginary axis, for $n\neq 0$, $1$, $2$, or $3$, all the terms in \eqref{ks} are positive.  For $n=2$ or $3$, we note that $0\leq \alpha_{n} < 1$, thus
\beq
\left(n-\mu \right)^2-1 + \alpha_{n}^2\left(\left(n-2-\mu\right)^2-1\right) > 4\left(n-\mu -1\right), \nonumber
\eeq   
which is positive for $n=2$ or $3$, $\mu \in [0,1)$.  

Likewise, along the negative imaginary axis, for $n\neq 0$, $1$, $-2$, or $-1$, all terms are positive.  For $n=-1$ or $-2$, we have
\beq
\left(n-\mu \right)^2-1 + \alpha_{n}^2\left(\left(n+2-\mu\right)^2-1\right) > -4\left(n-\mu +1\right), \nonumber
\eeq 
so that the eigenvalues corresponding to $n=-1$ and $-2$ on the negative imaginary axis have positive Krein signature.  

However, for $n=0$ or $1$ on either part of the imaginary axis, if we let $B \rightarrow 0^{+}$, $\alpha_{n} \rightarrow 0$.  Therefore we see that for sufficiently small $B$, with all other parameters fixed, the eigenvalues $\lambda_{\infty}(n) + \lambda_{p}(n)$ for $n=0$ or $1$ have negative Krein signature.  On the other hand, fixing all other parameters except $B$, if we allow the offset size $B$ to become arbitrarily large, then $\alpha_{n} \rightarrow 1$ and 
\beq
\lim_{B \rightarrow \infty }\left(n-\mu \right)^2-1 + \alpha_{n}^2\left(\left(n\pm2-\mu\right)^2-1\right) = 2(n-\mu \pm 1)^2. \label{lblim}
\eeq
Hence it is possible for eigenvalues to pass through the origin or switch Krein signature.  

Being more careful, we focus on the eigenvalues with potentially negative Krein signature, which are 
\beq
\lambda^{+}(1) = i\frac{k}{2}(2-\mu)\left(2k-\left(k^{2}(2-\mu)^{2}+4B\hat{r}_{2}\right)^{1/2} \right), \nonumber
\eeq 

\beq
\lambda^{-}(0) = i\frac{k}{2}(1+\mu)\left(-2k+\left(k^{2}(1+\mu)^{2}+4B\hat{r}_{-1}\right)^{1/2} \right), \nonumber
\eeq

\beq
\lambda^{-}(1) = i\frac{k\mu}{2}\left(-2k+\left(k^{2}\mu^{2}+4B\hat{r}_{0}\right)^{1/2} \right), \nonumber
\eeq
and
\beq
\lambda^{+}(0) = i\frac{k(1-\mu)}{2}\left(2k-\left(k^{2}(1-\mu)^{2}+4B\hat{r}_{1}\right)^{1/2} \right). \nonumber
\eeq

One sees that for given $k$ and $\epsilon$, one can find a sufficiently large value of $B$ such that none of these four eigenvalues pass through the origin for $\mu \in (0,1)$.  Noting that $\hat{r}_{n}=\hat{R}_{k,\epsilon}(n-\mu)>0$ by Hypothesis $H3^{'}$,  and since $\hat{R}_{k,\epsilon}(n-\mu)$ is continuous in $\mu$ (see Lemma \ref{propsaboutr}), we define $\tilde{r}_{n}=\min_{\mu\in[0,1]}\hat{R}_{k,\epsilon}(n-\mu)$.  We define the parameter $B^{\ast}$ to be 
\[
B^{\ast} = \max\left\{\frac{3k^{2}}{4\tilde{r}_{2}},\frac{3k^{2}}{4\tilde{r}_{-1}} ,\frac{k^{2}}{\tilde{r}_{0}},\frac{k^{2}}{\tilde{r}_{1}} \right\}.
\]
If $B > B^{\ast}$, then all four eigenvalues cannot pass through the origin for $\mu \in (0,1)$.  Setting $\mu = \frac{1}{2}$, one has from \eqref{lblim} that each of the four eigenvalues must have positive Krein signature for $B > B^{*}$.    
    
We now apply a theorem of \cite{hara} which states that  
\beq
k_{r} + k_{c} + k^{-}_{i} = n(L_{\mu}), 
\label{kreincondition}
\eeq
where $k_{r}$ is the number of eigenvalues of $JL_{\mu}$ on the positive real axis, $k_{c}$ is the number of eigenvalues with real part, $k^{-}_{i}$ is the number of imaginary eigenvalues with negative Krein signature, and $n(L_{\mu})$ is the number of negative eigenvalues of $L_{\mu}$.  In order to apply this theorem, one needs to show that the operator $JL_{\mu}$ satisfies Assumptions $2.1a-d$ in \cite{hara}.  Given that $JL_{\mu}=JL_{\mu,c} + 2JK$, where $JK$ is compact, and that the reciprocals of the eigenvalues of $JL_{\mu,c}$ are square summable, then showing all four assumptions hold for $JL_{\mu}$ is straightforward.  For $\mu =1/2$, $k_{r}=k_{c}=k^{-}_{i}=0$, and thus $n(L_{1/2})=0$.  Since the operator $L_{\mu}$ remains invertible for $\mu \in (0,1)$, which means no eigenvalue passes through the origin, then $n(L_{\mu})=0$ for $\mu \in (0,1)$.  This establishes that every eigenvalue of $JL_{\mu}$ has positive Krein signature.  
    
\subsection{Spectral and Orbital Stability for Small Potential Height}

As shown above, for $B$ sufficiently large, $V_{0}=0$, and $\mu \in (0,1)$, there are no eigenvalues of negative Krein signature, which by \eqref{kreincondition} implies that the operator $L_{\mu}$ is positive definite. Thus a standard perturbation argument guarantees that for small enough $V_{0}$, no eigenvalue of $L_{\mu}$ crosses through the origin, and thus we must have $n(L_{\mu}) = 0$.  Using \eqref{kreincondition} again shows that \eqref{sol} is spectrally stable for a given $\mu$ with sufficiently small potential height.  

In the case that $\mu = 0$, $V_{0} = 0$, one has by continuity of the spectrum with respect to the parameter $\mu$ that every eigenvalue of $JL_{0}$ must be on the imaginary axis.  However, for any value of $V_{0}$, there is an eigenvalue at the origin, with eigenvector 
\beq
\varphi_{nu} = \left(\ba{r}D\sin(x) \\ -\cos(x) \ea \right), \nonumber
\eeq
due to the phase symmetry which generates \eqref{sol}.  Thus \eqref{kreincondition} cannot be applied.  Likewise, there is a generalized eigenvector of $JL_{0}$ at the origin,
\beq 
\varphi_{gn} = \left(\ba{r} D\cos(x) \\ \sin(x) \ea \right). \nonumber
\eeq
Using the work above, one formally sees that the eigenvalues at the origin correspond to the eigenvalues $\lambda^{-}_{1}$ and $\lambda^{-}_{-1}$ colliding at the origin for $\mu=0$.  We now prove that the generalized kernel of $JL_{0}$ consists only of $\varphi_{mu}$ and $\varphi_{gn}$.  First define the projection operator (\cite{kato},Theorem 6.17) 
\[
P_{\mu} = \frac{1}{2\pi i}\oint_{\Gamma} \left(JL_{\mu} - \lambda \right)^{-1}d\lambda,
\]
where $\Gamma$ is a closed, bounded contour in the complex plane such that $\Gamma \cap \sigma(JL_{\mu})=\emptyset$ and the origin is inside $\Gamma$.  We further suppose that $\lambda^{-}_{1}$ and $\lambda^{-}_{-1}$ are the only eigenvalues of $JL_{\mu}$ inside $\Gamma$ for $\mu$ sufficiently small.  Since $JL_{0}$ has a compact resolvent, it has discrete eigenvalues that accumulate only at infinity.  Therefore, we can also choose $\Gamma$ such that $\Gamma \cap \sigma(JL_{0})=\emptyset$ and so that $\Gamma$ contains only a finite, counting multiplicity, number of the eigenvalues of $JL_{0}$.  Thus the projection $P_{0}$ is well-defined and finite-dimensional.  We then have 
\[
\gnorm{P_{\mu}-P_{0}}_{2,v} \leq \sup_{\lambda \in \Gamma} \gnorm{\left(JL_{\mu} - \lambda \right)^{-1} - \left(JL_{0} - \lambda \right)^{-1}}_{2,v}.
\]
Since $\gnorm{\left(JL_{\mu} - \lambda \right)^{-1} - \left(JL_{0} - \lambda \right)^{-1}}_{2,v}$ is continuous in $\lambda$, on $\Gamma$, which is compact, the supremum is attained.  We further restrict $\Gamma$ such that $\Gamma\cap \sigma(JL_{c,0})=\emptyset$.  Using Lemma \ref{convresolv} then gives
\[
\lim_{\mu \rightarrow 0^{+}} \gnorm{P_{\mu}-P_{0}}_{2,v} = 0.
\]  
Since $P_{\mu}$ and $P_{0}$ are projection operators, we then have (\cite{kato}, pg. 156) that 
\[
\mbox{dim}(\mbox{Ran}(P_{\mu}))=\mbox{dim}(\mbox{Ran}(P_{0})), 
\]
where $\mbox{Ran}(P_{\mu})$ denotes the range of $P_{\mu}$.  By construction $\mbox{dim}(\mbox{Ran}(P_{\mu}))=2$, and thus $\mbox{dim}(\mbox{Ran}(P_{0}))=2$.  The dimension of $\mbox{Ran}(P_{0})$ counts the algebraic multiplicity of an eigenvalue (see \cite{kato}, pg. 181), and so we see that the generalized kernel of $JL_{0}$ can only consist of $\varphi_{nu}$ and $\varphi_{gn}$.

Since $L_{0}$ is self adjoint and has a compact resolvent (see Lemma \ref{compresolv}), it cannot have a generalized eigenvalue at the origin for any $V_{0}$.  At $D=1$ ({\it i.e.} $V_{0}=0$), it is straightforward to show that 
\beq
L_{0}\varphi_{gn} = 2B \varphi_{gn}. \nonumber
\eeq
Thus, for $V_{0}=0$ and $B>0$, the operator $L_{0}$ has a simple eigenvalue at the origin and otherwise has only positive eigenvalues.  Since the eigenvector at the origin persists for any $V_{0}<0$, every nonzero eigenvalue of $L_{0}$ remains positive for small values of $V_{0}$.  This implies that  $n(L_{0})=0$ for $V_{0}<0$ and $|V_{0}|$ sufficiently small.  If $n(L_{0})=0$, one concludes spectral stability for small potential height by way of the following argument.  If $JL_{0}\varphi = \lambda \varphi$, $\lambda \neq 0$, then 
\beq
\left<L_{0}\varphi,\varphi \right> = -\lambda\left<J\varphi,\varphi \right>, \nonumber
\eeq
and $\left<L_{0}\varphi,\varphi \right> > 0$ since $\varphi$ is not in the kernel of $L_{0}$ by assumption.  $\left<J\varphi,\varphi \right>$ is strictly imaginary and nonzero.   Thus $\lambda$ is strictly imaginary.  We have now shown that the spectrum of $JL$ on $\mbox{D}(JL) = H^{2}_{2}(S_{2\pi n}) \subset L^{2}_{2}(S_{2\pi n})$, which is decomposed as
\beq
\sigma(JL) = \bigcup^{n-1}_{r=0} \sigma(JL_{r/n}), \nonumber
\eeq  
is strictly imaginary for small potential height $V_{0}$ since $\sigma(JL_{r/n})$ is strictly imaginary and there are a finite number of values $r$. 

As for orbital stability, again consider \eqref{sol} in the form $\psi(x,t) = \phi_{\omega}(x)e^{-i\omega t}$.  In other words, the solution is generated by the phase symmetry of the Hamiltonian problem \eqref{nlocgp} (see Section 2 for the explicit form of the Hamiltonian).  Returning to the original scaling whereby $V(x)$ is a $\frac{2\pi}{k}$ periodic function, with $\phi_{\omega}(x)$ a $\frac{2\pi}{k}$-periodic function, we pose the stability problem on $H^{2}_{1}(S_{T})$, where $T = \frac{2\pi n}{k}$, with $n\in \mathbb{N}$.  Thus we are working with \eqref{realandimagpart}. 

Then, again, we note that one has the conservation of the quantity 
\[
I(f,g) = \frac{1}{2}\int_{-\pi n/k}^{\pi n/k} \left(f^{2}(x) + g^{2}(x)\right)dx.
\]
If we denote the Hamiltonian as $\mathcal{H}(\psi_{r},\psi_{i})$, one has that $\phi_{\omega}=\phi_{\omega,r}+i\phi_{\omega,i}$ is a critical point of $E(f,g)$, where  
\beq
E(f,g) = \mathcal{H}(f,g) - \omega I(f,g). \nonumber
\eeq
One has $L=\mathcal{H}''(\phi_{\omega,r},\phi_{\omega,i})-\omega$, where the primes denote variational derivatives.  In the given scaling, $\mu \in [0,k)$, so that a perturbation of period $\frac{2 \pi n}{k}$ corresponds to $\mu = \frac{k}{n}$.  In this case, we showed that $L$ is positive semi-definite on $H^{2}_{2}(S_{T})\subset L^{2}_{2}(S_{T})$.  One can forgo the requirement that the function $d(\omega)=E(\phi_{\omega,r},\phi_{\omega,i})$ be convex (see \cite{gss}) and conclude orbital stability in the space $H_{1}(S_{T})$ by combining the stable two dimensional real solution into a complex function.

\section{Spectral Instability for Small Offset Size}
In contrast to the approach above, we equate the offset size $B$ to zero, for given $\mu$ and $\epsilon$.  We obtain the linearization 
\beq
J\left( 
\ba{cc}
L^{+}_{\mu} & 0  \\
0 & L^{-}_{\mu}
\ea
\right) ,
\nonumber
\eeq
with
\beq
L^{+}_{\mu} = -\frac{k^{2}}{2}\left(\left(\p_{x}+i\mu\right)^{2}+1\right), \nonumber
\eeq 
and
\beq
L^{-}_{\mu} = -\frac{k^{2}}{2}\left(\left(\p_{x}+i\mu\right)^{2}+1\right) + 2A\sin(x)R_{k,\epsilon,\mu}\ast(\sin(x) \cdot). \nonumber
\eeq
We introduced the scaling $x\rightarrow kx$, so that $\mu\in[0,1)$.  The linearized problem is in ``canonical form" (see \cite{hara}), and one has the following theorem from \cite{hara}.  
\begin{thm}\cite{hara}
 Let $n(L^{+}_{\mu})$ and $n(L^{-}_{\mu})$ denote the number of negative eigenvalues of the given operators.  With $k_{r}$ the number of real eigenvalues of the canonical system, one has
\[
k_{r} \geq \left|n(L^{+}_{\mu}) - n(L^{-}_{\mu})  \right|.
\]
\end{thm}
For $A$ small or $\epsilon$ large, the problem is spectrally stable since the problem is a small perturbation of the constant coefficient case.  Further, one has $n(L^{+}_{\mu})=2$.  However, with $\epsilon = 0$, $L^{-}_{0}$ is a special case of Hill's equation \cite{mag}.  We can use the standard spectral theory for Hill's equation, from which the spectrum of $L^{-}$ ({\it i.e.} the $\mu$ independent operator) is in the bands $[\gamma_{0},\gamma'_{1}] \cup [\gamma'_{2},\gamma_{1}]\cup \cdots$, where $\gamma_{j}$ is an eigenvalue of $L^{-}_{0}$ and $\gamma'_{j}$ is an eigenvalue for $\mu=\frac{1}{2}$, with the eigenvalues for all other values of $\mu$ filling in the bands continuously.  If we can establish that $L^{-}_{0}$ is positive definite, the same must hold for $L^{-}_{\mu}$, and we will have then shown instability for all values of $\mu$.  

Therefore, setting $\mu, ~ \epsilon=0$, we examine the quadratic form $\left<L^{-}_{0} f,f \right>$.  Let 
\[
f(x) = \sum_{j=-\infty}^{\infty}\hat{f}_{j}e_{j}(x), 
\]
with $e_{j}(x)$ from \eqref{fs} where $T=2\pi$.  One has
\beq
\left<L^{-}_{0} f,f \right> = \sum_{j=-\infty}^{\infty}\frac{k^{2}(j^{2}-1)}{2}\left|\hat{f}_{j} \right|^{2} + \frac{A}{2}\left|\hat{f}_{j}-\hat{f}_{j+2}\right|^{2}.
\label{qfsum}
\eeq
There is one negative direction corresponding to the $j=0$ mode.  Thus, we let $f = a e_{0} + b g(x)$, where $|a|^{2}+|b|^{2}=1$, and $g(x)$ is orthogonal to $e_{0}(x)=\frac{1}{2\pi}$.  It is straightforward to show that
\[
L^{-}_{0}e_{0} = \frac{1}{2\pi}\left(\frac{-k^2}{2}+2A\sin^{2}(x) \right) = (-\frac{k^2}{2}+A)e_{0} - \frac{A}{\sqrt{2\pi}}(e_{2}+e_{-2}).
\]
Therefore,
\[
\left<L^{-}_{0} f, f \right> = |a|^{2}(-\frac{k^2}{2}+A) + |b|^{2}\left<L^{-}_{0}g,g \right> - \frac{2A}{\sqrt{2\pi}}\mbox{Re}\left(a^{\ast}b(\hat{g}_{2}+\hat{g}_{-2}) \right), 
\]
and 
\beq
\left<L^{-}_{0}f,f \right> \geq |a|^{2}(-\frac{k^2}{2}+A) + |b|^{2}\left<L^{-}_{0}g,g \right> - \frac{2A}{\sqrt{2\pi}}|a||b||\hat{g}_{2}+\hat{g}_{-2}|. \nonumber
\eeq
Assume $A\geq\frac{k^2}{2}$, and define $c = -\frac{k^2}{2}+A$.  We rewrite the above inequality as
\beq
\left<L^{-}_{0}f,f \right> \geq \left(|a|\sqrt{c} - \frac{A}{\sqrt{2\pi c}}|b||\hat{g}_{2}+\hat{g}_{-2}| \right)^{2} + |b|^{2}\left<L^{-}_{0}g,g \right> - \frac{A^{2}}{2\pi c}|b|^{2}|\hat{g}_{2}+\hat{g}_{-2}|^{2}, \nonumber
\eeq
which leads us to examine
\beq
\left<L^{-}_{0}g,g \right> - \frac{A^{2}}{2\pi c} |\hat{g}_{2}+\hat{g}_{-2}|^{2}.
\nonumber
\eeq
With $\hat{g}_{0}=0$, we take those terms in $\eqref{qfsum}$ that involve only terms in $\hat{g}_{2}$ and $\hat{g}_{-2}$, which reduces our efforts to analyzing 
\beq
\frac{3k^{2}+A}{2}\left(|\hat{g}_{-2}|^{2} + |\hat{g}_{2}|^{2}\right) - \frac{A^{2}}{2\pi c}|\hat{g}_{2}+\hat{g}_{-2}|^{2}.  \nonumber
\eeq  
Using Young's inequality, $|\hat{g}_{2}+\hat{g}_{-2}|^{2}\leq 2\left(|\hat{g}_{-2}|^{2} + |\hat{g}_{2}|^{2}\right)$.  It follows that if we can satisfy the inequality
\beq
\frac{3k^{2}+A}{2} - \frac{A^{2}}{\pi c} > 0,
\label{Acond}
\eeq
we prove that $\left<L^{-}_{0}g,g \right> > 0$, and the problem is unstable.  A straightforward computation shows that \eqref{Acond} is satisfied if
\beq
A \geq \frac{2\left(-1+\sqrt{4-\frac{6}{\pi}}\right)}{1-\frac{2}{\pi}}k^2  \approx 2.4533k^2 . \nonumber
\eeq
For the sake of presentation we write the instability condition as $A \geq 2.46 k^2$.  
\section{Numerics}
In this section, we present numerical results applied to situations for which we expect our theoretical results to apply.  Then, after calibrating our numerics in this sense, we present numerical experiments that correspond to the work in \cite{decon}.  

For all the results shown, a filtered pseudo-spectral method \cite{guo} is used for the spatial variable, while MATLAB's ODE45 function was used for the integration in time.  The specific filtering function used is $\sigma(x)=e^{\alpha x^{2\gamma}}$, where $\alpha = \log(\mbox{eps})$, with eps denoting machine precision, and $\gamma = 4$.  Again see \cite{guo} for more details and analysis.  

For the figures in this paper, 128 modes on the domain $[0,8\pi]$ were used in the pseudo-spectral approximation; higher mode runs were tested and gave identical results to those using 128 modes.  In each figure, a perturbation of the form
\beq
\nu m(x)e^{i\theta(x)} \nonumber
\eeq   
was added to the initial condition $\phi_{\omega}(x)$.  The function $m(x)$ is a randomly generated, $8 \pi$ periodic function, normalized so that $\left|\left|m(x)\right|\right|_{2}=1$, while $\nu$ is typically $.01$.  However, in certain cases consider $\nu=.1$, and where this is the case, it is noted.  Finally, given the identities derived for the convolution in this paper, the convolution integral turns into a simple term by term multiplication of two vectors in the pseudo-spectral method.  Thus no approximations to the integral or kernel are made.  As in \cite{decon}, we convolve against $\phi(x) = e^{-x^{2}}$.  In every figure, $\alpha$ and $k$ are one.    
  
Figure \ref{fig:B_small} shows the results for $B=.01$, $V_{0}=-2.46$, with nonlocality parameter $\epsilon=0$.  As expected from Theorem 4, we see an instability emerge with these parameter values with the random perturbation to the initial condition as explained previously.  In contrast, Figure \ref{fig:B_big} shows the case $B=1$, $V_{0}=-.01$, with nonlocality parameter $\epsilon = .01$, $\nu = .01$.  The numerics behave as the Theorems 2 and 3 predict.  We have confidence that the numerical results are accurate and correspond to the existing theory.

Figures \ref{fig:B_big_nu_small} and \ref{fig:B_big_nu_big} show the results for the case $B=1$, $V_{0}=-1$, $\epsilon=.01$, $\nu = .01$ and $.1$.  This is a direct comparison to the work in \cite{decon}.  As can be seen from the figures, the underlying solution appears to be robust to perturbations, even with a nonzero nonlocality parameter.  This contradicts the results of \cite{decon}, and seems to imply that \eqref{sol} is stable in this parameter regime.    

\section{Conclusion}
We have shown that for a large class of kernels, $R(x,\epsilon)$, used to represent long range nonlocal interactions in a Gross-Pitaevskii equation, if one lets the range of nonlocality go to zero, {\it i.e.} $\epsilon \rightarrow 0$, then in a rigorous sense, the wavefunction for the nonlocal problem approaches the wavefunction for the local problem.  This result holds for any smooth periodic trapping potential $V(x)$.  Thus we have demonstrated that generalizing a local model to a nonlocal one can be done in a straightforward way, thus expanding the modeling potential  of Gross-Pitaevskii equations.  Likewise, we have established the stability properties of a particular class of solutions to a nonlocal Gross-Pitaevskii equation.  The theory and numerical experiments predict that when the offset size $B$ is large these solutions are stable.  It is therefore possible that under the right conditions these solutions could be observed as wavefunctions describing a Bose-Einstein condensate.  

\section*{Acknowledgments}
The author would like to thank M. Ablowitz, B. Deconinck, A. Rey, and H. Segur for reading through earlier versions of this manuscript and for their insightful advice and comments.  The author would especially like to thank the authors of \cite{decon}, B. Deconinck and J.N. Kutz, for their encouragement, interest, and endorsement of the results in this work which was expressed through private communications.  The  author finally would like to thank the reviewer for helping to make this a significantly better paper.   
\begin{center}
\begin{figure}[!hp]
   \subfloat[$B=.01$, $V_{0}=-2.46$, $\epsilon=0$, $\nu=.01$]{\label{fig:B_small}\includegraphics[scale=.4]{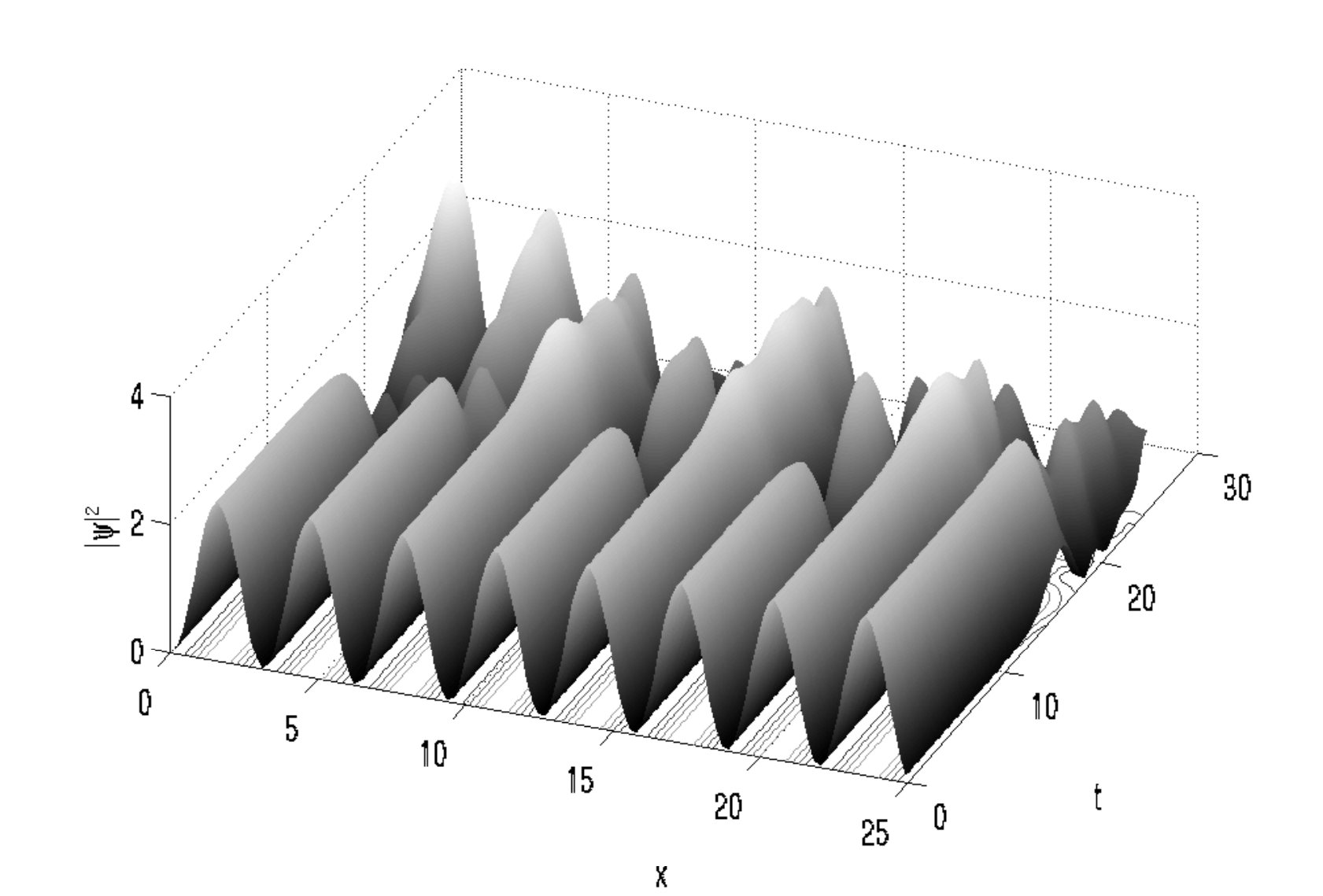}}
   \subfloat[$B=1$, $V_{0}=-.01$, $\epsilon=.01$, $\nu=.01$]{\label{fig:B_big}\includegraphics[scale=.4]{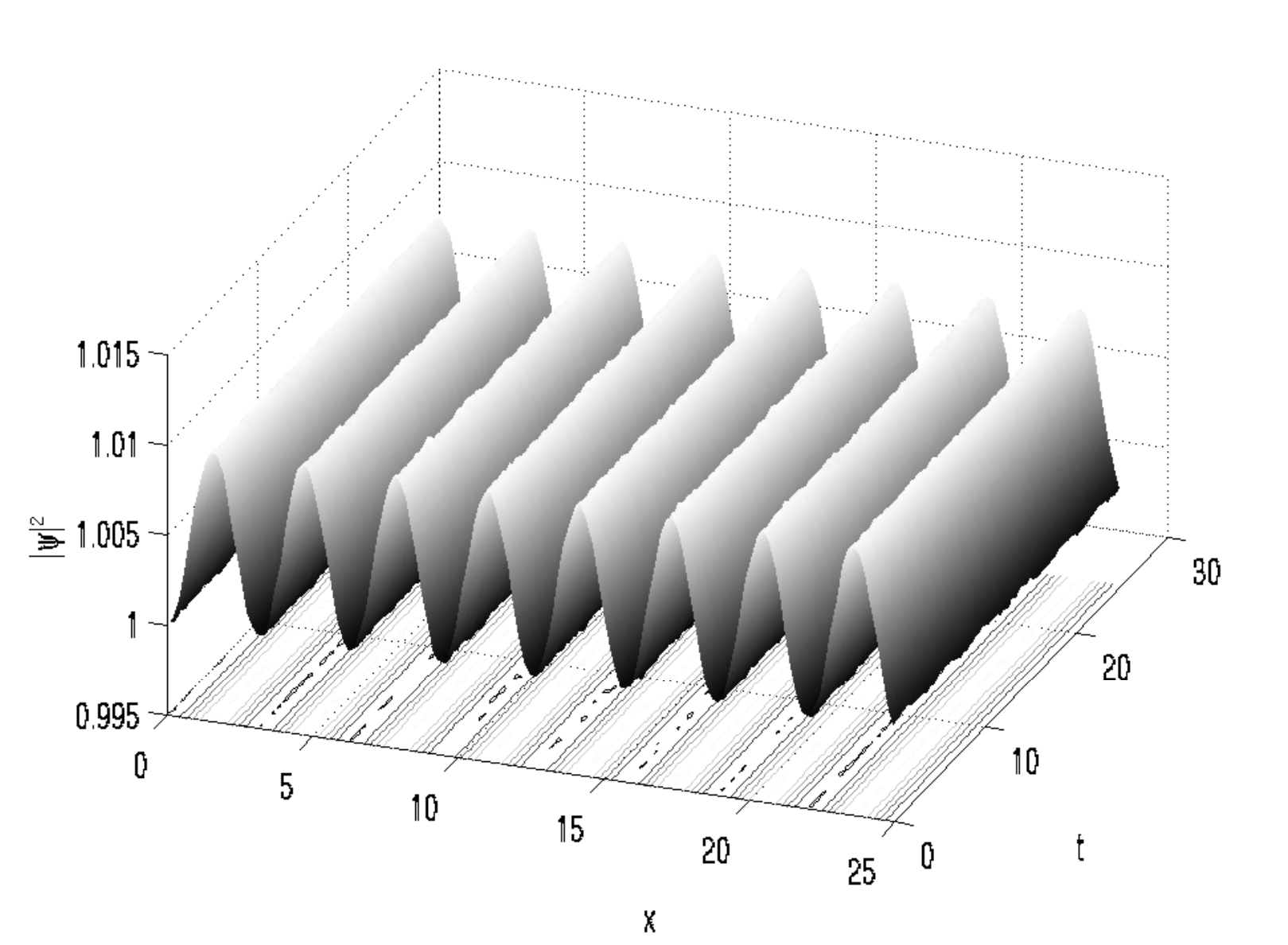}}
\caption{Confirmation of Theorems 2 and 3}
\end{figure}

\begin{figure}[!hp]
   \subfloat[$B=1$, $V_{0}=-1$, $\epsilon=.01$, $\nu=.01$]{\label{fig:B_big_nu_small}\includegraphics[scale=.4]{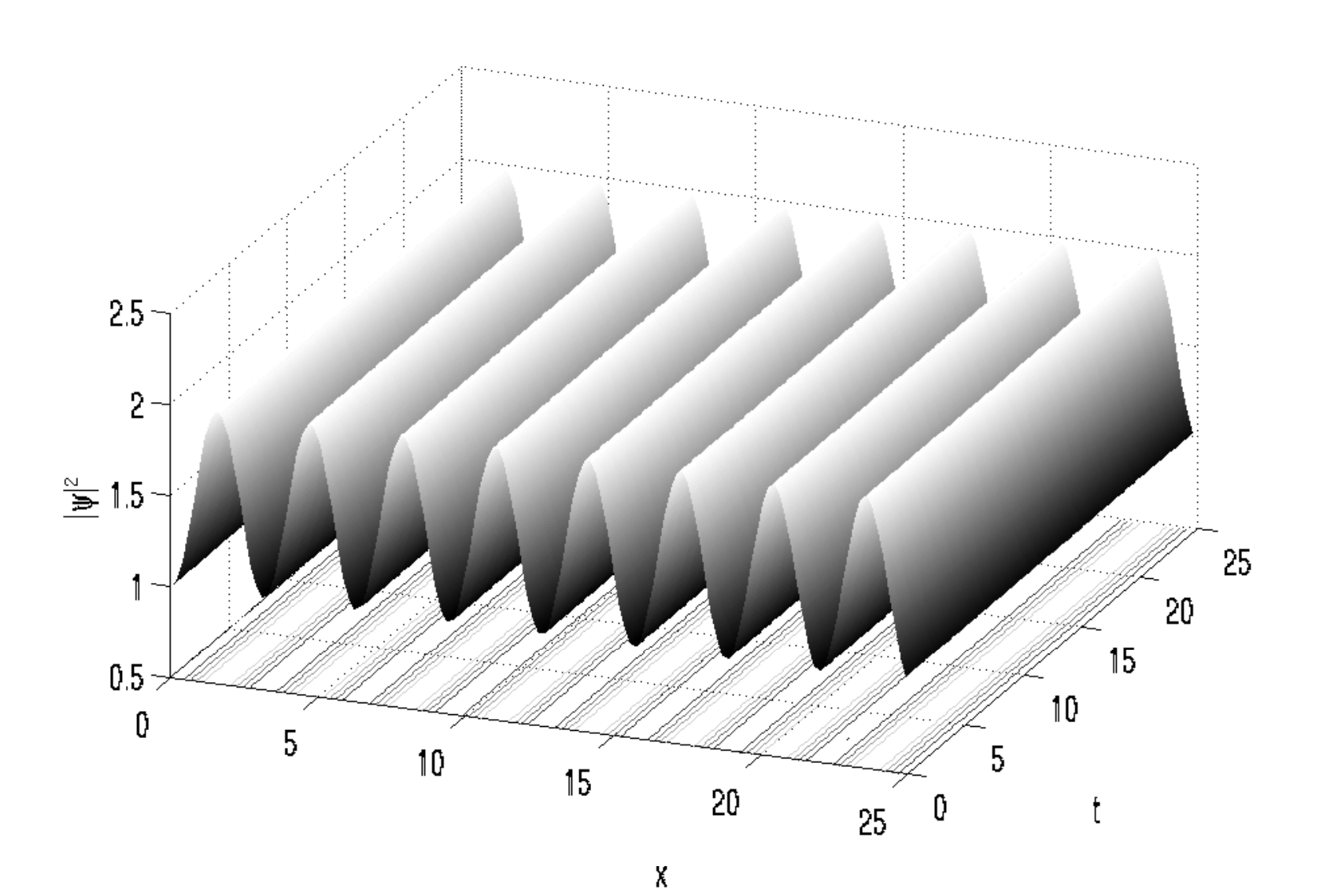}}
    \subfloat[$B=1$, $V_{0}=-1$, $\epsilon=.01$, $\nu=.1$]{\label{fig:B_big_nu_big}\includegraphics[scale=.4]{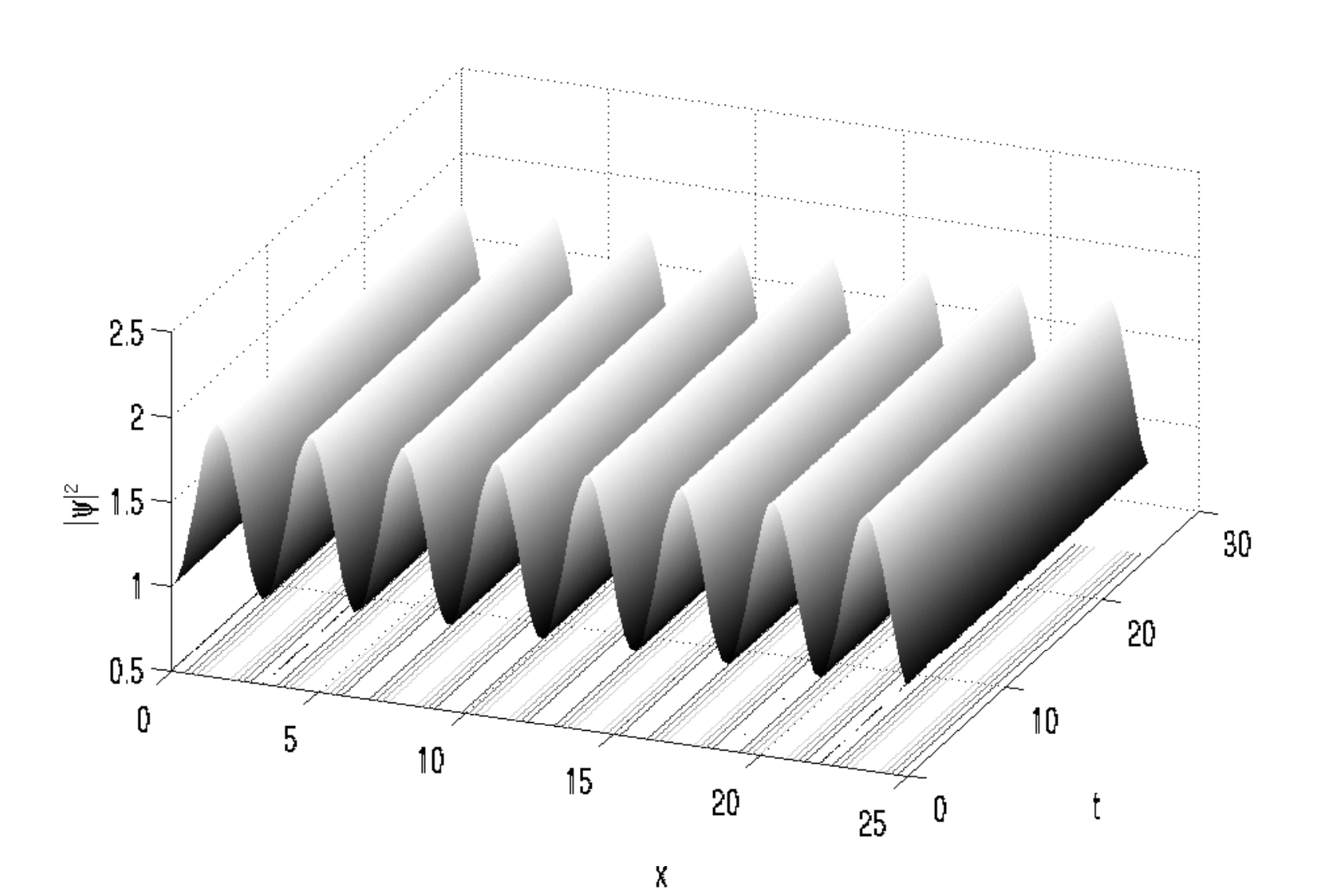}}
\caption{Numerical Predictions for Large Offset Size and Potential Height}
\end{figure}
\end{center}
\pagebreak

\end{document}